\documentclass[leqno,oneside,12pt]{amsart} 
\usepackage{pict2e} 
\usepackage{amssymb}
\usepackage[english]{babel}

\usepackage[textsize=footnotesize,textwidth=20ex,colorinlistoftodos]{todonotes}

\usepackage{hyperref}
\hypersetup{
    colorlinks=true,
    linkcolor=blue,
    filecolor=magenta, 
    citecolor=red,     
    urlcolor=cyan,
}

\usepackage[capitalize]{cleveref}
    \crefname{enumi}{}{}
    \Crefname{enumi}{Item}{Items}
    \crefname{equation}{}{}
    \Crefname{equation}{Equation}{Equations}    

\theoremstyle{plain} 
\newtheorem{theorem}{\indent\sc Theorem}[section]
\newtheorem{lemma}[theorem]{\indent\sc Lemma}
\newtheorem{corollary}[theorem]{\indent\sc Corollary}
\newtheorem{proposition}[theorem]{\indent\sc Proposition}

\theoremstyle{definition} 
\newtheorem{definition}[theorem]{\indent\sc Definition}
\newtheorem{remark}[theorem]{\indent\sc Remark}
\newtheorem{example}[theorem]{\indent\sc Example}

\usepackage[active]{srcltx}
\usepackage{pdfsync}
\usepackage{tabularx} 
\usepackage{arydshln}

\newcommand{\B}{\mathcal{B}}
\newcommand{\CC}{\mathbb{C}}

\newcommand{\FF}{\mathbb{F}}
\newcommand{\ZZ}{\mathbb{Z}}

\newcommand{\OO}{\mathbb{O}}
\newcommand{\RR}{\mathbb{R}}

\renewcommand{\hom}{\mathrm{Hom}}
\newcommand{\ii}{\textbf{i}}
\def\Aut{\mathop{\rm Aut}}

\def\dim{\mathop{\hbox{\rm dim}}}
\def\tr{\mathop{\rm tr}}
  
\newcommand{\spf}{\mathfrak{sp}}

\newcommand{\suf}{\mathfrak{su}}
\newcommand{\sof}{\mathfrak{so}}
\newcommand{\slf}{\mathfrak{sl}}

\newcommand{\ad}{\mathop{\mathrm{ad}}}
\newcommand{\id}{\mathrm{id}}

\newcommand{\g}{\mathfrak{g}}

\renewcommand{\a}{\alpha}

\title[Inner ideals of real simple Lie algebras]{Inner ideals of real simple Lie algebras}  

\author[C.~Draper]{ 
Cristina Draper${}^*$ 
}  

\author[J.~Meulewaeter]{Jeroen Meulewaeter${}^\star$}


\thanks{${}^*$ Supported by    Junta de Andaluc\'{\i}a  through projects  FQM-336, UMA18-FEDERJA-119, and PAIDI project P20\_01391, and  by the Spanish Ministerio de Ciencia e Innovaci\'on   through projects  PID2019-104236GB-I00 and PID2020-118452GB-I00, all of them with FEDER funds.}
\thanks{${}^\star$ Supported by  a PhD Fellowship of the Research Foundation Flanders (Belgium) (F.W.O.- Vlaanderen), 166032/1128720N}

\address[C. Draper]{Departamento de Matem\'{a}tica Aplicada,  Escuela de Ingenier\'\i as Industriales, 
Universidad de M\'{a}laga, 
 29071 M\'{a}laga \\Spain}
\email{cdf@uma.es}
\address[J. Meulewaeter]{Department of Mathematics: Algebra and Geometry, Ghent University, Krijgslaan 281--S25, 9000 Gent\\Belgium}
\email{jeroen.meulewaeter@ugent.be}

\begin{document}


\maketitle 

 
\begin{abstract}
 A classification up to automorphism of the inner ideals of the real finite-dimensional simple Lie algebras is given, 
 jointly with precise descriptions in the case of the exceptional Lie algebras.
 \end{abstract}

\section{Introduction }     

Inner ideals in Lie algebras were introduced in \cite{Faulkner1973}, where Faulkner  defines inner ideals of modules for Lie algebras. This notion came from   Jordan algebras' world.   Faulkner's aim was not to develop   a general theory of inner ideals, but he used inner ideals to reconstruct in some way the geometry: a hexagonal geometry with certain collineations is coordinatized by a Jordan division algebra which permits to construct the Lie algebra. We come back later to this idea.

Georgia Benkart in her doctoral thesis \cite{tesisgeorgia} began a systematic research  of the role of the inner ideals, as well as the ad-nilpotent elements, in the study of Lie algebras. 
This topic was suggested to her by Seligman,
taking into account the relevance for the theory of Jordan algebras of studying those ones satisfying the minimum condition on Jordan inner ideals,  obviously  inspired by   associative algebras  satisfying the descending chain condition on its left or right ideals.  
Her following papers \cite{inner76,transactions77}  layed the   groundwork for the development of an Artinian theory for Lie algebras.

A considerable amount of material of inner ideals of Lie algebras, including the results in \cite{ transactions77}, can be found in the recent monograph by the AMS \cite{libroAntonio}, which   tries to show how Jordan theory can be applied to the study of Lie algebras not necessarily of finite dimension.  
Precisely the book contains  a whole chapter about an Artinian theory for Lie algebras \cite[Chapter~12]{libroAntonio}. After introducing the notion of a complemented inner ideal, the following  key result is stated: every abelian inner ideal $B$ of finite length of a non-degenerate Lie algebra $L$ yields a finite $\mathbb Z$-grading $L=L_{-n}\oplus \dots\oplus L_n$ such that $B=L_n$ so that in particular $B$ is a complemented inner ideal. In fact, a Lie algebra is complemented (meaning that every inner ideal is so)  if and only if it is a direct sum of simple non-degenerate Artinian Lie algebras.  Note that the algebras considered here not only have arbitrary dimension,    they can even be  considered over a ring of scalars with very few restrictions.

It is well known the rich interplay of inner ideals with ad-nilpotent elements. Using inner ideals, some generalizations of Kostrikin's lemma have been achieved in \cite{Miguel2009}, which finds ad-nilpotent elements of index 3 from ad-nilpotent elements of greater index. These kind of elements   play a fundamental role in the classification of simple modular Lie algebras. They have strong implications on the structure of the algebra.
For instance, for algebraically closed fields (of characteristic greater than 5) the existence of an ad-nilpotent element implies that the finite-dimensional non-degenerate simple Lie algebra is necessarily classical \cite{transactions77}. Thus we have 
 criterion for distinguishing some concrete Lie algebras from other ones using inner ideals or ad-nilpotent elements.  
 Similarly 
 \cite{Baranov2013} obtained a new characterization of infinite locally finite simple diagonal Lie algebras (characteristic zero)  in terms of inner ideals. 
 \smallskip

Apart from this algebraic line of study of inner ideals, from which we have only included    a small sample of the contributions,  
 there is a second line of study,  focused on the more geometric Faulkner's approach, which follows being topical.  
It is based on incidence geometries. The story begins with Tits, who introduced  buildings  in \cite{Tits74}   in order to study algebraic groups. Some point-line spaces were associated to these buildings, called root shadows \cite[Definition~3]{Cohen2006}. Many of these ones coincide with the following point-line spaces associated to Lie algebras: the so called   \emph{extremal geometry} of a Lie algebra $L$ has as points the points in the projective space $P(L)$ spanned by extremal elements in $L$, and as lines  the lines in $P(L)$ such that all their points are  spanned by extremal elements. 
In  the case that $L$ is finite-dimensional, simple,   and generated (as an algebra) by pure extremal elements, if the set of lines is non-empty, then the extremal geometry  is 
 a  root shadow space (for instance see \cite{Cohen2007}).
Very recently \cite{Cohen21} 
extends Faulkner's results showing that the correspondence  between inner ideals of the Lie algebra of a simple algebraic group  and  shadows  on  the  set  of  long  root  groups  of  the  building  associated  with  the  algebraic group  holds for fields of characteristic different from  two too. 

This raises the question whether one can recover the Lie algebra  from its extremal geometry. Is the algebra characterized by its geometry? A work in this direction is 
 \cite{Cuypers2018}, which proves that the simple Lie algebra   is  uniquely determined (up to
isomorphism)  when the  extremal geometry  is the root shadow space of a spherical building of rank at least 3. (This requires that the algebra is generated by its set of extremal elements, but, for characteristic distinct from 2 and from 3 and finite dimension, it is enough  that it contains an extremal element that is not a sandwich \cite{Cohen2008}).

Other authors try to   focus on direct connections between Lie algebras and buildings, without the intermediate step of considering algebraic groups, see for instance,  \cite{Medts2020,Cuypers2020}. The idea is to use structurable algebras, a class of non-associative algebras with involution generalizing Jordan algebras, introduced in \cite{Allison1979}, which is usually  employed to construct 5-graded Lie algebras via
the Tits-Kantor-Koecher construction. So, Moufang polygons are constructed in  \cite{Medts2020}   using inner ideals of Lie algebras obtained from structurable algebras via such construction. 
The considered geometry in \cite{ Cuypers2020} is not the extremal  one but a generalization called   the \emph{inner ideal geometry} of a Lie algebra, where 
  the points are the 1-dimensional inner ideals but the lines are the minimal proper inner ideals containing at least two points. It turns out to be   a Moufang spherical building of rank at least $2$, or  a Moufang set in case there are no lines.  \smallskip

The main purpose in this paper is to deal with the case of the real algebras. Note that a great amount of the mentioned works  aim to avoid restrictions on the characteristic of the field. So, in a sense the real case is one of the easiest cases, after the complex one.  Concretely, real simple finite-dimensional Lie algebras are completely classified and they are very  well-known. But on the other hand, they are very important cases, that are worth a specific study. We cannot forget their fundamental role in Differential Geometry and in Physics, which not even needs references.    We concentrate on two objectives. First, in this work, on the classification up to isomorphism of the inner ideals of the real simple finite-dimensional Lie algebras, which is achieved in Theorem~\ref{mainth}, and  is followed by a case by case description in \cref{section_Explicit classification}. Second, in a  forthcoming  paper, what  are we able to say about the incidence geometries related to these inner ideals?   
To study the real case will allow us to gain intuition about the inner ideal geometries, often difficult to   visualize.  Besides, our study will provide some concrete realizations in \cref{corolarioexc} of the exceptional real Lie algebras obtained from structurable algebras related to composition algebras, joint with other realizations focused on finding point line spaces in \cref{sec_ps}.

Our first tool  to address the classification is our knowledge  on the classification of inner ideals in the complex case. 
Both cases, real and complex,  are obviously related, since an inner ideal of a real Lie algebra will be, after complexification, an inner ideal of the complexified algebra.  Then we use the complete description of the   abelian inner ideals of any simple Lie algebra $L$ over $\CC$   obtained in \cite{Draper2012} (of course it can   also be consulted in \cite{libroAntonio}, which collects many results).
Namely, 
take $H$ a Cartan subalgebra of $L$, and fix $\Delta=\{\alpha_1,\dots,\alpha_l\}$ a basis of simple roots of the root system $\Phi$ relative to $H$.
Denote by $\tilde\alpha=\sum_{i=1}^lm_i\alpha_i$ the maximal root in  $\Phi$. 
Then, for any  non-zero abelian inner ideal $B$ of $L$, there is an automorphism $\varphi\in\Aut(L)$ and a subset $I\subseteq\{1,\dots,l\}$ such that $\varphi(B)=B_I$, where,
 $$
B_I:=\bigoplus_{\alpha\in \Phi}\{L_\alpha  :
\alpha=\sum_{1\le i\le l} p_i \alpha_i\hbox{ with } p_j=m_j
\hbox{ for all } j\in I\}.
$$

In order  to relate the inner ideals of the real Lie algebras with the inner ideals of the complex ones, we will use the Satake diagrams of the corresponding real forms. We will obtain a similar result than in the complex case, that is, that every abelian inner ideal is conjugated to some $\mathcal{B}_I$, but only for $I$ a subset of indices 
  \emph{adapted to the Satake diagram}  (Definition~\ref{de_adapted}).
 In some sense, it can be considered as a generalization of   \cite{RealesZgrads}, which proves that there is a bijection between isomorphism classes of real semisimple $\mathbb Z$-graded Lie algebras and weighted (with certain restrictions) Satake diagrams. \smallskip

The structure of the paper is as follows. Preliminaries are recalled in  \cref{se_prelim}: the Satake diagrams of the simple real Lie algebras
and the classification of the  inner ideals of the (also simple) complex Lie algebras. The main result  in the real case is \cref{mainth} in Section~\ref{se_mainresult}. This theorem is easily applied in Section~\ref{section_Explicit classification} to obtain a detailed classification of the inner ideals of the real Lie algebras up to automorphism. This turns to be a combinatorial description, so
we have added Section~\ref{se_moreonexceptional}, which tries to describe the inner ideals of the exceptional Lie algebras without using roots, but in terms of some constructions of these algebras. To be precise, we use the Tits-Kantor-Koecher construction applied to the tensor product of composition algebras and to Albert algebras, as well as some nice models of the split exceptional algebras based entirely on linear and multilinear algebra, which allow us to obtain some remarkable inner ideals formed completely by extremal elements.
As a byproduct, all the  non-compact exceptional Lie algebras other than those of type   $G_2$ are constructed by TKK-construction of the tensor product of real composition algebras in  \cref{corolarioexc}.
\smallskip

 \textbf{Acknowledgement} The second author would like to thank the first author for her hospitality during a research visit to Malaga in October-November 2019. Also,  the first author thanks Alberto Elduque for fruitful conversations on incidence geometries during a research visit to Zaragoza in October 2021, and Antonio Fern\'andez for his suggestions on Jordan pairs.


\section{Preliminaries}\label{se_prelim}

All the algebras considered throughout this paper are finite-dimensional (over a field).

\subsection{Background on Satake diagrams}\label{se_Satake}

We recall some very well known facts on Satake diagrams in order to fix notation, mainly extracted from \cite[Chapter III, \S7]{Helgason2001} and from the summary in \cite{DraperFontanals2016}.

Let $\mathfrak{g}$ be a simple Lie algebra over $\mathbb{R}$ and $\kappa\colon \mathfrak{g}\times\mathfrak{g}\to\mathbb{R}$ its Killing form.

Take any maximal abelian subspace $\mathfrak{a}$ such that $\kappa\vert_{\mathfrak{a}\times\mathfrak{a}}$ is positive definite. (So there exists a Cartan involution $\theta\colon\mathfrak{g}\to\mathfrak{g}$ such that $\theta\vert_{\mathfrak{a}}=-\id_{\mathfrak{a}}$.)
For each $\lambda$ in the dual space $\mathfrak{ a}^*$ of $\mathfrak{ a}$, let
$\mathfrak{ g}_{\lambda}=\{x\in\mathfrak{ g}:[h,x]=\lambda(h)x , \forall h\in\mathfrak{ a}\}$. Then $\lambda$
is called a \emph{restricted root} if $\lambda\ne0$ and $\mathfrak{ g}_{\lambda}\ne0$. 
Denote by $\Sigma$ the set of restricted roots, which is an abstract root system (not necessarily reduced), and by $m_\lambda=\dim \mathfrak{g}_\lambda$ the \emph{multiplicity} of the restricted root.
Note that the simultaneous diagonalization of $\ad_\mathfrak{ g}\mathfrak{ a}$ gives the decomposition $\mathfrak{ g}=\mathfrak{g}_0\oplus\big(\displaystyle\oplus_{\lambda\in\Sigma}\mathfrak{g}_\lambda\big)$, for $\mathfrak{g}_0=\mathfrak{a}\oplus \mathop{\rm Cent}_{\mathfrak{t}}(\mathfrak{a})$, with $\mathfrak{t}=\mathrm{Fix}(\theta)$.

Now take $\mathfrak{h}$  any maximal abelian subalgebra of $\mathfrak{g}$ containing $\mathfrak{a}$. Then $\mathfrak{h}$ is a Cartan subalgebra of $\mathfrak{g}$ (that is, $\mathfrak{h}^\mathbb{C}$ is a Cartan subalgebra of $\mathfrak{g}^\mathbb{C}$). 
Denote by $\Phi$   the root system of $\mathfrak{g}^\mathbb{C}$ relative to $\mathfrak{h}^\mathbb{C}$ and by $\mathfrak{g}^\CC_\alpha$ the one-dimensional root space for any $\alpha\in \Phi$.
If $\alpha\in \Phi$, denote by $\bar\alpha:=\alpha\vert_{\mathfrak{a}}\colon\mathfrak{a}\to\mathbb{R}$. 
The roots  in $\Phi_0=\{\alpha\in\Phi:\bar \alpha=0\}$ are called the \emph{compact} roots and those in $\Phi\setminus\Phi_0$ the \emph{non-compact} roots. 
Note that $\alpha\in\Phi_0$ if and only if $\alpha(\mathfrak{h})\subseteq\ii\mathbb{R}$.  
The restricted roots are exactly the  non-zero restrictions of roots to $\mathfrak{a}\subseteq \mathfrak{h}^\mathbb{C} $, that is,
  $\Sigma=\{\bar \alpha:\alpha\in\Phi\setminus\Phi_0\}$. 
Moreover, for any $\lambda\in\Sigma$,
 \begin{equation}\label{eq_espaciosraices}
  \mathfrak{g}_\lambda=(\oplus\{\mathfrak{g}^\CC_\alpha:\bar\alpha=\lambda\})\cap \mathfrak{g},
\end{equation}
where we understand $\mathfrak{g}$ to be naturally contained in $\mathfrak{g}^\CC$;
and $m_\lambda$ coincides with the number of roots $\alpha\in\Phi$ satisfying $\bar\alpha=\lambda$.
Moreover,  it is possible to choose a basis $\Delta$  of the root system $\Phi$ in such a way that
$\Delta_0={\Delta} \cap {\Phi_0}=\{\alpha\in \Delta:\bar\alpha=0\}$ is a basis of    $\Phi_0$, also a root system.

The Satake diagram of the real Lie algebra $\mathfrak{g}$ is defined as follows.
In the Dynkin diagram associated to such basis $\Delta$, the roots in $\Delta_0$ are denoted by a black circle $\bullet$ and the roots in $\Delta\setminus \Delta_0$ are denoted by a white circle $\circ$.
If $\alpha,\beta\in \Delta\setminus \Delta_0$ are such that $\bar \alpha=\bar \beta$, then $\alpha$ and $\beta$ are joined by a curved arrow.  The \emph{real rank} of $\g$ is defined as $\dim\mathfrak a$, which coincides the number of white nodes in the Satake diagram minus the number of arrows.

\begin{remark}
 If $\mathfrak{g}$ is compact, it turns out that $\mathfrak{a}=0$, so $\Phi_0=\Phi$ and necessarily
 its Satake diagram is   the Dynkin diagram (of $\mathfrak{g}^\mathbb{C}$)   with all the nodes colored in black. 
 If $\mathfrak{g}$ is split, it turns out that $\mathfrak{a}$ is a Cartan subalgebra ($\mathfrak{a}=\mathfrak{h}$), so  $\Phi_0=\emptyset$ and necessarily
 its Satake diagram is   the Dynkin diagram (of $\mathfrak{g}^\mathbb{C}$)   with all the nodes in white. 
 \end{remark}

\subsection{Satake diagrams of the simple real Lie algebras}\label{se_losdiagramasconcretos}
We recall the classification in order to unify notation and labellings. Here $I_n$ denotes the identity matrix, $I_{p,q}:=\mathrm{diag}(I_p,-I_q)$ and $J:=\tiny\begin{pmatrix}0&I_n\\-I_n&0\end{pmatrix}$. \ \smallskip

 $\bullet$ \textbf {Special type.}
  The real forms of the special linear algebra $\mathfrak{sl}_{n+1}(\CC)$ of the traceless matrices are 
  \begin{itemize}
   \item[$\circ$] the split Lie algebra, $\mathfrak{sl}_{n+1}(\RR)$, of real rank $n$;  
    \item[$\circ$] $\mathfrak{sl}_m(\mathbb H)=\{x\in\mathfrak{gl}_{m}(\mathbb H):\textrm{Re}(\tr(x))=0\}$,  only for odd $n=2m-1>1$, 
   with real rank $m-1$ and Satake diagram \vspace{-10pt}
   \begin{center}{ 
\begin{picture}(23,5)(4,-0.5)  
\put(5,0){\circle*{1}} \put(9,0){\circle{1}} \put(13,0){\circle*{1}}
    \put(25,0){\circle{1}}\put(29,0){\circle*{1}}
\put(5.5,0){\line(1,0){3}}
\put(9.5,0){\line(1,0){3}} \put(13.5,0){\line(1,0){3}}  
\put(17.5,0){\circle{0.1}} \put(18.5,0){\circle{0.1}} \put(19.5,0){\circle{0.1}} \put(20.5,0){\circle{0.1}} 
\put(21.5,0){\line(1,0){3}}
\put(25.5,0){\line(1,0){3}} 
\end{picture}
}\end{center}\vspace{6pt}

  \item[$\circ$]   $\mathfrak{su}_{p,q}=\{x\in\mathfrak{sl}_{n+1}(\CC):I_{p,q}x+\bar x^t I_{p,q}=0\}$, with $p+q=n+1$, $p\le q$,
  with  real rank $p$ and  Satake diagram 
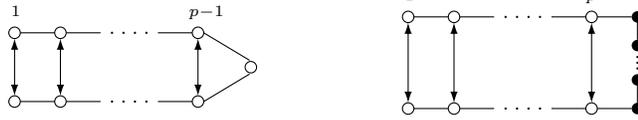
\vspace{-10pt}
  \begin{center}{
\begin{picture}(28,5)(4,-0.5)  
 \put(9,0){\circle{1}} \put(13,0){\circle{1}}\put(25,0){\circle{1}}  
 \put(8.7,1.7){$\scriptstyle _1$}  \put(24.2,1.7){$\scriptstyle _{p-1}$}  
   \put(29.6,-3){\circle{1}}  \put(9,-6){\circle{1}} \put(13,-6){\circle{1}}\put(25,-6){\circle{1}}  
 \put(9.5,0){\line(1,0){3}} \put(13.5,0){\line(1,0){3}} \put(17.5,0){\circle{0.1}} \put(18.5,0){\circle{0.1}} \put(19.5,0){\circle{0.1}} \put(20.5,0){\circle{0.1}} \put(21.5,0){\line(1,0){3}} \put(25.5,0){\line(2,-1.2){4}} 
\put(25.5,-6){\line(2,1.2){4}}
 \put(9.5,-6){\line(1,0){3}} \put(13.5,-6){\line(1,0){3}} \put(17.5,-6){\circle{0.1}} \put(18.5,-6){\circle{0.1}} \put(19.5,-6){\circle{0.1}} \put(20.5,-6){\circle{0.1}} \put(21.5,-6){\line(1,0){3}} 
\put(9,-0.75){\line(0,-1){4.5}} \put(9,-1.5){\vector(0,1){1}}\put(9,-4.5){\vector(0,-1){1}} 
\put(13,-0.75){\line(0,-1){4.5}} \put(13,-1.5){\vector(0,1){1}}\put(13,-4.5){\vector(0,-1){1}}
\put(25,-0.75){\line(0,-1){4.5}} \put(25,-1.5){\vector(0,1){1}}\put(25,-4.5){\vector(0,-1){1}}
\end{picture}
\qquad
\begin{picture}(23,5)(4,-1.9)  
\put(8.7,1.7){$\scriptstyle _1$}  \put(24.5,1.7){$\scriptstyle _{p}$}  
 \put(9,0){\circle{1}} \put(13,0){\circle{1}}\put(25,0){\circle{1}} \put(29,0){\circle*{1}}\put(29,-2.5){\circle*{1}}\put(29,-5.5){\circle*{1}}\put(29,-4.5){\circle{0.1}}\put(29,-3.5){\circle{0.1}}\put(29,-4){\circle{0.1}}
 \put(9,-8){\circle{1}} \put(13,-8){\circle{1}}\put(25,-8){\circle{1}}  
 \put(9.5,0){\line(1,0){3}} \put(13.5,0){\line(1,0){3}} \put(17.5,0){\circle{0.1}} \put(18.5,0){\circle{0.1}} \put(19.5,0){\circle{0.1}} \put(20.5,0){\circle{0.1}} \put(21.5,0){\line(1,0){3}}\put(25.5,0){\line(1,0){3}}  \put(29,0){\line(0,-1){2}} 
\put(29,-8){\line(0,1){2}} 
 \put(9.5,-8){\line(1,0){3}} \put(13.5,-8){\line(1,0){3}} \put(17.5,-8){\circle{0.1}} \put(18.5,-8){\circle{0.1}} \put(19.5,-8){\circle{0.1}} \put(20.5,-8){\circle{0.1}} \put(21.5,-8){\line(1,0){3}} \put(25.5,-8){\line(1,0){3}} \put(29,-8){\circle*{1}} 
\put(9,-0.75){\line(0,-1){6.5}} \put(9,-1.5){\vector(0,1){1}}\put(9,-6.5){\vector(0,-1){1}}  
\put(13,-0.75){\line(0,-1){6.5}} \put(13,-1.5){\vector(0,1){1}}\put(13,-6.5){\vector(0,-1){1}}
\put(25,-0.75){\line(0,-1){6.5}} \put(25,-1.5){\vector(0,1){1}}\put(25,-6.5){\vector(0,-1){1}}
\end{picture}
}\end{center}\vspace{26pt}
if $p=q$  
and $p<q$ respectively.
Here we assume $p\ge1$, for $p=0$ the algebra  $\mathfrak{su}_{0,n+1}\equiv \mathfrak{su}_{n+1}$ is the  compact one.

  \end{itemize}

  $\bullet$  \textbf {Orthogonal type.}
   The real forms of  the orthogonal algebras $ \mathfrak{so}_{2n+1}(\CC)$ and $\mathfrak{so}_{2n}(\CC) $ of skew-symmetric matrices are
   \begin{itemize}
    \item[$\circ$]  $\mathfrak{so}_{p,q}(\RR)\equiv \mathfrak{so}_{p,q}=\{x\in\mathfrak{gl}_{p+q}(\mathbb R):I_{p,q}x+ x^t I_{p,q}=0\}$, with $p\le q$. If $p+q=2n+1$ (type $B_n$), the Satake diagram is:  
    \begin{center}{ 
\begin{picture}(29,5)(4,-2.5)  
\put(5,0){\circle{1}}  \put(17,0){\circle{1}}\put(21,0){\circle*{1}}\put(33,0){\circle*{1}}\put(37,0){\circle*{1}}
\put(4.7,-2){$\scriptstyle _1$}  \put(16.5,-2){$\scriptstyle _{p}$} 

\put(25.5,0){\circle{0.1}} \put(26.5,0){\circle{0.1}} \put(27.5,0){\circle{0.1}} \put(28.5,0){\circle{0.1}} 
\put(9.5,0){\circle{0.1}} \put(10.5,0){\circle{0.1}} \put(11.5,0){\circle{0.1}} \put(12.5,0){\circle{0.1}} 
   \put(5.5,0){\line(1,0){3}}
\put(13.5,0){\line(1,0){3}} \put(17.5,0){\line(1,0){3}} \put(21.5,0){\line(1,0){3}} 
\put(29.5,0){\line(1,0){3}}
\put(33.5,0.15){\line(1,0){3}}\put(33.5,-0.15){\line(1,0){3}}
\put(34.7,0.8){\line(1,-1){0.8} } 
\put(34.7,-0.8){\line(1,1){0.8} }
\end{picture}
}\end{center}  
Here there are $0\le p\le n$ white nodes ($p$ is the  real rank).  For $p=n$ we have the split real form, and if $p=0$ the compact one.
      
    If $p+q=2n$ (type $D_n$), the Satake diagrams are
     
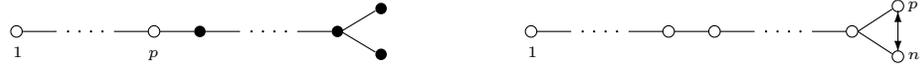
\begin{center}{     
     \begin{picture}(40,5)(8,-0.5)  
      \put(4.7,-2){$\scriptstyle _1$}  \put(16.5,-2){$\scriptstyle _{p}$}  
\put(5,0){\circle{1}}  \put(17,0){\circle{1}}\put(21,0){\circle*{1}}\put(33,0){\circle*{1}}
\put(36.8,2){\circle*{1}}\put(36.8,-2){\circle*{1}}
\put(25.5,0){\circle{0.1}} \put(26.5,0){\circle{0.1}} \put(27.5,0){\circle{0.1}} \put(28.5,0){\circle{0.1}} 
\put(9.5,0){\circle{0.1}} \put(10.5,0){\circle{0.1}} \put(11.5,0){\circle{0.1}} \put(12.5,0){\circle{0.1}} 
   \put(5.5,0){\line(1,0){3}}
\put(13.5,0){\line(1,0){3}} \put(17.5,0){\line(1,0){3}} \put(21.5,0){\line(1,0){3}} 
\put(29.5,0){\line(1,0){3}}
\put(33.3,0){\line(2,1.2){3}} \put(33.3,0){\line(2,-1.2){3}} 
\end{picture}
\begin{picture}(20,5)(4,-0.5)  
\put(4.7,-2){$\scriptstyle _1$}  \put(37.85,2.2){$\scriptstyle _{p}$} \put(37.85,-2.2){$\scriptstyle _{n}$}  
\put(5,0){\circle{1}}  \put(17,0){\circle{1}}\put(21,0){\circle{1}}\put(33,0){\circle{1}}
\put(37,2.2){\circle{1}}\put(37,-2.2){\circle{1}} 
\put(25.5,0){\circle{0.1}} \put(26.5,0){\circle{0.1}} \put(27.5,0){\circle{0.1}} \put(28.5,0){\circle{0.1}} 
\put(9.5,0){\circle{0.1}} \put(10.5,0){\circle{0.1}} \put(11.5,0){\circle{0.1}} \put(12.5,0){\circle{0.1}} 
   \put(5.5,0){\line(1,0){3}}
\put(13.5,0){\line(1,0){3}} \put(17.5,0){\line(1,0){3}} \put(21.5,0){\line(1,0){3}} 
\put(29.5,0){\line(1,0){3}}
\put(33.5,0){\line(2,1.3){3}} \put(33.5,0){\line(2,-1.3){3}}  
\put(37,1.8){\line(0,-1){3.5}} \put(37,-1){\vector(0,-1){0.8}}\put(37,1){\vector(0,1){0.8}} 
\end{picture}
}\end{center}\vskip0.4cm  
if $p\le n-2$ and $p=n-1$  
 respectively, and again the real rank is $p$. For $p=n$ we have the split real form, and if $p=0$ the compact one.

  \item[$\circ$]  $\mathfrak{so}^*_{2n}(\RR)\equiv \mathfrak{u}^*_{n}(\mathbb H)=\{x\in\mathfrak{gl}_{n}(\mathbb H):x^th+ h\bar x=0\}$, where $h=\mathrm{diag}(\ii,\dots,\ii)=\ii I_n$.
  This case  (type $D_n$) only happens when $n>4$. The Satake diagrams are     
  
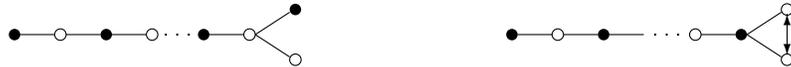
\begin{center}{  
\begin{picture}(40,5)(10,-0.5)  
\put(12.5,0){\circle*{1}}  \put(16.5,0){\circle{1}}\put(20.5,0){\circle*{1}}\put(24.5,0){\circle{1}}\put(29,0){\circle*{1}}\put(33,0){\circle{1}}
\put(37,2.2){\circle*{1}}\put(37,-2.2){\circle{1}} 
\put(25.7,0){\circle{0.1}} \put(26.7,0){\circle{0.1}} \put(27.7,0){\circle{0.1}}  
\put(13,0){\line(1,0){3}} \put(17,0){\line(1,0){3}} \put(21,0){\line(1,0){3}} 
\put(29.5,0){\line(1,0){3}}
\put(33.5,0){\line(2,1.3){3}} \put(33.5,0){\line(2,-1.3){3}}  
\end{picture}
\begin{picture}(20,5)(8,-0.5)  
\put(13,0){\circle*{1}}  \put(17,0){\circle{1}}\put(21,0){\circle*{1}}\put(29,0){\circle{1}}\put(33,0){\circle*{1}}
\put(37,2.2){\circle{1}}\put(37,-2.2){\circle{1}}  
\put(25.5,0){\circle{0.1}} \put(26.5,0){\circle{0.1}} \put(27.5,0){\circle{0.1}}  
\put(13.5,0){\line(1,0){3}} \put(17.5,0){\line(1,0){3}} \put(21.5,0){\line(1,0){3}} 
\put(29.5,0){\line(1,0){3}}
\put(33.5,0){\line(2,1.3){3}} \put(33.5,0){\line(2,-1.3){3}}  
\put(37,1.8){\line(0,-1){3.5}} \put(37,-1){\vector(0,-1){0.8}}\put(37,1){\vector(0,1){0.8}} 
\end{picture}
}\end{center}\vskip0.4cm  
if $n$ is even and odd, respectively. (The real rank is the integer part of $n/2$.)

    \end{itemize}

    $\bullet$  \textbf {Symplectic type.}
    The real forms of  $ \mathfrak{sp}_{2n}(\CC)$ are     
    \begin{itemize} 
   \item[$\circ$]   The split Lie algebra $\mathfrak{sp}_{2n}(\RR)=\{x\in\mathfrak{gl}_{2n}(\mathbb R): Jx+  x^tJ=0\}$, of real rank $n$;  
    \item[$\circ$]    $\mathfrak{sp}_{p,q}(\mathbb H)\equiv \mathfrak{sp}_{p,q}=\{x\in\mathfrak{gl}_{n}(\mathbb H):I_{p,q}x+\bar x^t I_{p,q}=0\}$, for $p\le q $   and $p+q=n$, 
    has real rank   $p$ and its Satake diagram is
    \begin{center}{ 
\begin{picture}(45,5)(4,-1.5)  
   \put(17,0){\circle{1}}\put(21,0){\circle*{1}}\put(33,0){\circle*{1}}\put(37,0){\circle*{1}}  
   \put(3.7,-2){$\scriptstyle _1$}  \put(16.5,-2){$\scriptstyle _{2p}$} \put(36.5,-2){$\scriptstyle _{n}$}  
\put(25,0){\circle*{1}} 
\put(26.5,0){\circle{0.1}} \put(27.5,0){\circle{0.1}} \put(28.5,0){\circle{0.1}} 
\put(4,0){\circle*{1}}\put(8,0){\circle{1}}\put(12,0){\circle*{1}}  
\put(13.5,0){\circle{0.1}} \put(14.5,0){\circle{0.1}} \put(15.5,0){\circle{0.1}} 
   \put(4.5,0){\line(1,0){3}} \put(8.5,0){\line(1,0){3}}
 \put(17.5,0){\line(1,0){3}} \put(21.5,0){\line(1,0){3}} 
\put(29.5,0){\line(1,0){3}}
\put(33.5,0.15){\line(1,0){3}}\put(33.5,-0.15){\line(1,0){3}}
\put(34.7,0){\line(1,1){0.8} } 
\put(34.7,-0){\line(1,-1){0.8} }
\end{picture}
\begin{picture}(30,5)(4,-1.5)  
   \put(13,0){\circle*{1}} \put(17,0){\circle{1}}\put(21,0){\circle*{1}}\put(33,0){\circle*{1}}\put(37,0){\circle{1}}\put(29,0){\circle{1}}
  \put(26.5,0){\circle{0.1}} \put(27.5,0){\circle{0.1}} \put(25.5,0){\circle{0.1}} 
  \put(12.7,-2){$\scriptstyle _1$}  \put(36.7,-2){$\scriptstyle _{2p}$}
   \put(13.5,0){\line(1,0){3}}
 \put(17.5,0){\line(1,0){3}} \put(21.5,0){\line(1,0){3}} 
\put(29.5,0){\line(1,0){3}}
\put(33.5,0.15){\line(1,0){3}}\put(33.5,-0.15){\line(1,0){3}}
\put(34.7,0){\line(1,1){0.8} } 
\put(34.7,-0){\line(1,-1){0.8} }
\end{picture}
}\end{center}\vskip0.4cm
if $1\le p<q$ or $p=q$, respectively ($p=0$ is the compact one).
  \end{itemize}

$\bullet$  \textbf {Exceptional type.}
The real forms of the exceptional complex Lie algebras, apart from the split and compact cases, have the next Satake diagrams

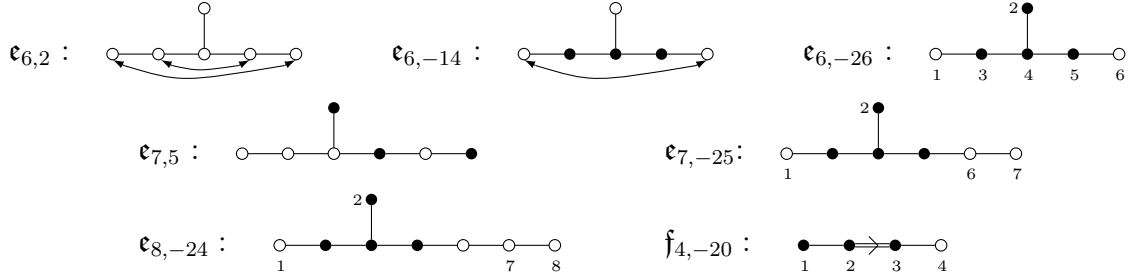
\begin{center}     
\begin{tabular}{lll}
$\mathfrak{e}_{6,2}$ :\quad \begin{picture}(23,5)(4,-0.5)
\put(5,0){\circle{1}} \put(9,0){\circle{1}} \put(13,0){\circle{1}}
\put(17,0){\circle{1}} \put(21,0){\circle{1}}
\put(13,4){\circle{1}} 
\put(5.5,0){\line(1,0){3}}
\put(9.5,0){\line(1,0){3}} \put(13.5,0){\line(1,0){3}}  
\put(17.5,0){\line(1,0){3}}
\put(13,0.5){\line(0,1){3}}
\cbezier (6,-1)(11,-2.5)(15,-2.5)(20,-1)  \put(20,-1){\vector(2,1){1}}
 \put(6,-1){\vector(-2,1){1}}
 \cbezier (10.2,-1)(12.5,-1.5)(14.5,-1.5)(16.0,-1) 
 \put(16.0,-1){\vector(2,1){1}}
 \put(10.2,-1){\vector(-2,1){1}}
 \end{picture}&
$\mathfrak{e}_{6,-14}$ :\quad \begin{picture}(23,5)(4,-0.5)
\put(5,0){\circle{1}} \put(9,0){\circle*{1}} \put(13,0){\circle*{1}}
\put(17,0){\circle*{1}} \put(21,0){\circle{1}}
\put(13,4){\circle{1}}
\put(5.5,0){\line(1,0){3}}
\put(9.5,0){\line(1,0){3}} \put(13.5,0){\line(1,0){3}}  
\put(17.5,0){\line(1,0){3}}
\put(13,0.5){\line(0,1){3}}
\cbezier (6,-1)(11,-2.5)(15,-2.5)(20,-1) 
 \put(20,-1){\vector(2,1){1}}
 \put(6,-1){\vector(-2,1){1}}
\end{picture}&
$\mathfrak{e}_{6,-26}$ :\quad \begin{picture}(23,5)(4,-0.5)  
\put(5,0){\circle{1}}\put(4.7,-2){$\scriptstyle _{1}$}
 \put(9,0){\circle*{1}}\put(8.7,-2){$\scriptstyle _{3}$}
  \put(13,0){\circle*{1}}\put(16.7,-2){$\scriptstyle _{5}$}\put(20.7,-2){$\scriptstyle _{6}$}
  \put(12.7,-2){$\scriptstyle _{4}$}
\put(17,0){\circle*{1}} 
\put(21,0){\circle{1}}  
\put(13,4){\circle*{1}} \put(11.4,3.8){$\scriptstyle _{2}$}
\put(5.5,0){\line(1,0){3}}
\put(9.5,0){\line(1,0){3}} \put(13.5,0){\line(1,0){3}}  
\put(17.5,0){\line(1,0){3}} 
\put(13,0.5){\line(0,1){3}}
\end{picture} 
\\
\end{tabular}\end{center}\vspace{5pt}

\begin{center}   
\begin{tabular}{ll}
$\mathfrak{e}_{7,5} $ :\quad \begin{picture}(23,5)(4,-0.5)  
\put(5,0){\circle{1}} \put(9,0){\circle{1}} \put(13,0){\circle{1}}
\put(17,0){\circle*{1}} \put(21,0){\circle{1}}\put(25,0){\circle*{1}} 
\put(13,4){\circle*{1}}
\put(5.5,0){\line(1,0){3}}
\put(9.5,0){\line(1,0){3}} \put(13.5,0){\line(1,0){3}}  
\put(17.5,0){\line(1,0){3}} 
\put(21.5,0){\line(1,0){3}}
\put(13,0.5){\line(0,1){3}}
\end{picture}\hspace{30pt}\ 
&
$\mathfrak{e}_{7,-25} $:\quad \begin{picture}(23,5)(4,-0.5)  
\put(5,0){\circle{1}} \put(9,0){\circle*{1}} \put(13,0){\circle*{1}}
\put(17,0){\circle*{1}} \put(21,0){\circle{1}}\put(25,0){\circle{1}} 
\put(13,4){\circle*{1}} \put(11.4,3.8){$\scriptstyle _{2}$}
\put(5.5,0){\line(1,0){3}}
\put(9.5,0){\line(1,0){3}} \put(13.5,0){\line(1,0){3}}  
\put(17.5,0){\line(1,0){3}} 
\put(21.5,0){\line(1,0){3}}
\put(13,0.5){\line(0,1){3}}
\put(4.7,-2){$\scriptstyle _{1}$}\put(20.7,-2){$\scriptstyle _{6}$}\put(24.7,-2){$\scriptstyle _{7}$}
\end{picture}\vspace{8pt}
\\
$\mathfrak{e}_{8,-24}:\quad$  
\begin{picture}(23,5)(4,-0.5)  
\put(5,0){\circle{1}} \put(9,0){\circle*{1}} \put(13,0){\circle*{1}}
\put(17,0){\circle*{1}} \put(21,0){\circle{1}}\put(25,0){\circle{1}}\put(29,0){\circle{1}}
\put(13,4){\circle*{1}}
\put(5.5,0){\line(1,0){3}}
\put(9.5,0){\line(1,0){3}} \put(13.5,0){\line(1,0){3}}  
\put(17.5,0){\line(1,0){3}} 
\put(21.5,0){\line(1,0){3}}
\put(25.5,0){\line(1,0){3}} 
\put(13,0.5){\line(0,1){3}}
\put(4.7,-2){$\scriptstyle _{1}$}\put(28.7,-2){$\scriptstyle _{8}$}\put(24.7,-2){$\scriptstyle _{7}$}\put(11.4,3.8){$\scriptstyle _{2}$}
\end{picture}\hspace{40pt}\ 
&
 
$\mathfrak{f}_{4,-20}:$  
\quad\begin{picture}(23,5)(4,-0.5)  
\put(5,0){\circle*{1}} \put(9,0){\circle*{1}} \put(13,0){\circle*{1}}
\put(17,0){\circle{1}}  
\put(5.5,0){\line(1,0){3}}
\put(9.5,0.15){\line(1,0){3}} \put(13.5,0){\line(1,0){3}} 
\put(10.7,0.8){\line(1,-1){0.8} } 
\put(10.7,-0.8){\line(1,1){0.8} }
\put(9.5,-0.15){\line(1,0){3}}
\put(4.7,-2){$\scriptstyle _{1}$}\put(8.7,-2){$\scriptstyle _{2}$}\put(12.7,-2){$\scriptstyle _{3}$}\put(16.7,-2){$\scriptstyle _{4}$}
\end{picture}\\
\end{tabular}\end{center}\vskip0.4cm
Here the second subindex, after the rank,  indicates the signature of the Killing form.
The real ranks are $4$, $2$ and $2$, for the real forms of $E_6$-type (respectively), $4$ and $3$, for the real forms of $E_7$-type, $\mathfrak{e}_{8,-24}$ has real rank $4$ and $\mathfrak{f}_{4,-20}$ equal to $1$. These real  ranks will be related with the maximal length of a chain of proper inner ideals as well as with the rank of the incidence geometries related to their inner ideals.

 
 \subsection{Inner ideals of complex simple Lie algebras}\label{se_background_inner}
 
 Usually the next definitions are considered over arbitrary fields, in spite that in this work we are mainly interested in the real field as an aim and in the complex field as a tool.

 \begin{definition}
\label{def_inner ideal}
A vector subspace $B$ of a Lie algebra $L$ is called an \textit{inner ideal} if $[B,[B,L]]\subseteq B$.
An inner ideal $B$   is called \textit{proper} if $B$ is neither $\{0\}$ nor $L$.
Such $B$   is called a \textit{minimal inner ideal} if it does not contain properly any  non-zero inner ideal.
\end{definition}

According to \cite[Lemma 1.13]{transactions77}, every proper inner ideal of a non-degenerate  simple finite-dimensional Lie algebra over a
field  of characteristic $0$ is necessarily abelian. The elements in an abelian inner ideal are ad-nilpotent, since 
  $[e,[e,[e,L]]]\subseteq[e,B]\subseteq[B,B]=0$ for each $0\ne e\in B$. Moreover, recall a well-known but key result.
  
  \begin{lemma}\label{le_lodeladescomposicion}
  Let $B$ be a proper inner ideal   of a simple (finite-dimensional) Lie algebra $L$ over a field of zero characteristic.   Every  $0\ne e\in B$ is ad-nilpotent of index 3 and 
  there exist $f$ and $h$ in $L$ such that $\{e,h,f\}$ is a $\mathfrak{sl}_2$-triple, that is,
$$
[h,e]=2e,\quad [e,f]=h,\quad [h,f]=-2f;
$$
and such that $\ad h$ diagonalizes $L$ with eigenvalues $\pm2$, $\pm1$ and $0$.
  \end{lemma}
  
  \begin{proof}
  As above, $e$ is ad-nilpotent, since $B$ is abelian.
  By the Jacobson-Morozov Lemma (for instance, consult the version in \cite[Chapter~III, \S11, Theorem~17]{libroJaconsonLie}), we can embed $e$ in a $\mathfrak{sl}_2$-triple of $L$. 
  In particular the nilpotency index of $e$ is precisely $3$ since $[e,[e,f]]=-2e\ne0$. Now,  
   \cite[Lemma~1]{Jacobson1958} gives the result on the eigenvalues of $\ad h$ and the corresponding diagonalization.
  \end{proof}
  
  An important source of inner ideals are precisely the ad-nilpotent elements.
It is clear that if $\langle e\rangle$ is an  inner ideal ($[e,[e,L]]\subseteq\langle e\rangle$), then $e$ is an ad-nilpotent element of index at most $3$.
 The converse is not true, but   if $e$ is an ad-nilpotent element of index at most $3$, then $[e,[e,L]]$ is an inner ideal (\cite[Lemma~1.8]{transactions77}).
 These elements play a key role in the incidence geometries related to inner ideals.

 \begin{definition}
\label{def_extremal element}
A  non-zero element $e$ of a Lie algebra $L$ is called an \textit{extremal element} if $[e,[e,L]]\subseteq\langle e\rangle$, that is, if $\langle e\rangle$ is an inner ideal.
Also, $e$  is called an \textit{absolute zero divisor} or a \textit{sandwich element} if $[e,[e,L]]=0$. 
An extremal element in $L$ is called   \textit{pure} if it is not an absolute zero divisor. A Lie algebra is \textit{non-degenerate} if it has no   absolute zero divisors. 
\end{definition}

\begin{example}
The elements  $e=\tiny\begin{pmatrix}0&1\\0&0\end{pmatrix}$ and $f=\tiny\begin{pmatrix}0&0\\1&0\end{pmatrix}$ in  the special  Lie algebra $\mathfrak{sl}_2(\RR)$ are both extremal elements.\end{example}

It is   clear, independently of the characteristic of the field, that $\langle e\rangle$ is a minimal abelian inner ideal for any  extremal element $e$. 
If besides $L$ is a simple Lie algebra over a field of zero characteristic, Lemma~\ref{le_lodeladescomposicion} says that every extremal element in $L$ is pure and 
that the Lie algebra $L$ is non-degenerate. (The non-degeneracy of the simple Lie algebra is even true by removing the hypothesis of finite dimension  \cite[Corollary~3.24]{libroAntonio}.)
 \medskip

Next we focus in the real case. If $B$ is an inner ideal of a real Lie algebra $\g$, then the complexification $B^\CC:=B\otimes_{\RR}\CC$ is an inner ideal of the complex Lie algebra $\g^\CC$. So, the first step in order to classify inner ideals of the real Lie algebras is to know the classification of the inner ideals of the complex Lie algebras. This classification is  well-known for complex finite-dimensional simple Lie algebras. In such case every abelian inner ideal coincides with the \emph{corner} $L_n$ of some $\ZZ$-grading $L=\oplus_{m=-n}^nL_m$ of the simple $\CC$-algebra $L$.
A very concrete description of the   abelian inner ideals of   $L$   is obtained in \cite{Draper2012} in terms of roots (and also related to Jordan pairs, see \cref{se_JP}). 
Namely, take $H$ a Cartan subalgebra of $L$, and fix $\Delta=\{\alpha_1,\dots,\alpha_l\}$ a basis of   the root system $\Phi$ relative to $H$.
Denote by $\tilde\alpha=\sum_{i=1}^lm_i\alpha_i$ the maximal root in  $\Phi$ relative to the fixed ordering. 
Consider, for any subset of indices $I\subseteq \{1,\dots,l\}$, the set of roots
\begin{equation}\label{eq_fiI}
\Phi_I:=\{\alpha=\sum_{1\le i\le l} p_i \alpha_i\in\Phi: p_j=m_j
\hbox{ for all } j\in I\},
\end{equation}
and the corresponding sum of root subspaces
\begin{equation}\label{eq_B_I}
B_I:=\bigoplus_{\alpha\in \Phi_I}L_\alpha=\bigoplus_{\alpha\in \Phi}\{L_\alpha\ :\,
\alpha=\sum_{1\le i\le l} p_i \alpha_i\hbox{ with } p_j=m_j
\hbox{ for all } j\in I\}.
\end{equation}
It is easy to check that $B_I$ is always an abelian inner ideal of $L$, but, as in \cite[Theorem~3.1 or Theorem~4.4]{Draper2012}, the converse is also true: 
  for any  non-zero abelian inner ideal $B$ of $L$, there is an automorphism $\varphi\in\Aut(L)$ and a subset $I\subseteq\{1,\dots,l\}$ such that $\varphi(B)=B_I$.

Note that, if $I\subseteq J\subseteq\{1,\dots,l\}$, then $B_J\subseteq B_I$.
In particular, the maximal abelian inner ideals  of $L$ are conjugate to $B_{\{i\}}$ for some $i\in\{1,\dots,l\}$. 
It is   straightforward to describe these inner ideals by means of combinatorial arguments, as well as the  lattice $\{B_I: I\subseteq\{1,\dots,l\}\}$ (for fixed $H$ and $\Delta$). The reader can find the related Hasse diagrams  in \cite{Draper2012}.


\section{Results on the real case}\label{se_mainresult}

Assume we   are still  in the above setting: $\g$ is a real simple Lie algebra with complexification $\g^\CC=L$,   $ \Delta=\{\alpha_1,\dots,\alpha_l\}$ is a set of simple roots of the root system $\Phi$ relative to  a Cartan subalgebra $H$ of $L$, and $\tilde\alpha=\sum_{i=1}^lm_i\alpha_i$ is the maximal root in  $\Phi$ relative to the ordering given by $\Delta$.

The following  result provides us a source for finding abelian inner ideals of $\mathfrak{g}$  whose complexification is just $B_I$ for a suitable choice of the indices in $I$.

\begin{definition}\label{de_adapted}
We will say that a non-empty  set $   I\subseteq\{1,\dots,l\}$  is \emph{adapted to the Satake diagram} of $\g$ if it satisfies\begin{itemize}
\item[i)] If $i\in I$, the related node in the Satake diagram of $\mathfrak{g}$ is white (that is, $\alpha_i\vert_{\mathfrak{a}}\ne0$);
\item[ii)] If $i\in I$, $j\in\{1,\dots,l\}$ and $\alpha_i\vert_{\mathfrak{a}}=\alpha_j\vert_{\mathfrak{a}}$, then $j\in I$.
\end{itemize}
\end{definition}

\begin{proposition}\label{pr_BI}
For any $I\subseteq\{1,\dots,l\}$ adapted to the Satake diagram of $\g$, we have
\begin{itemize}
\item[a)] If $\alpha\in \Phi_I$, then $\alpha\vert_{\mathfrak{a}}\in\Sigma$;
\item[b)] If $\alpha\in \Phi_I$ and $\beta\in\Phi$ satisfies $\alpha\vert_{\mathfrak{a}}=\beta\vert_{\mathfrak{a}}$, then $\beta\in  \Phi_I$;
\item[c)]  $\mathcal{B}_I:=\oplus\{\mathfrak{g}_{\bar\alpha}: \alpha\in\Phi_I\}$ is an abelian inner ideal of $\mathfrak{g}$ such that $(\mathcal{B}_I)^\CC=B_I$.
\end{itemize}
\end{proposition}

\begin{proof} Recall that, if $r=\dim\mathfrak{a}$ (the  {real rank} of $\mathfrak{g}$), the set $\{\bar\alpha_i:\alpha_i\in\Delta\setminus\Delta_0\}$ is a set of $r$ linearly independent elements in $\mathfrak{a}^*$. 
Note that 
\[r=l-|\Delta_0|-\frac 1 2 r_0,\] 
with $r_0$ the number of nodes connected by an arrow in the Satake diagram.
Indeed, this follows by taking into account that $\bar\alpha_i=\bar\alpha_j$ if $\alpha_i$ and $\alpha_j$ are connected by an arrow.

Part a) is clear, since if $\alpha=\sum_{i=1}^l p_i \alpha_i\in\Phi_I$, then the coefficient of $\bar\alpha_j$ in $\bar\alpha =\sum_{\alpha_i\notin\Delta_0} p_i \bar\alpha_i$  for $j\in I$ is $p_j=m_j\ne0$. 
For b), take roots $\alpha=\sum_{i=1}^l p_i \alpha_i$, $\beta=\sum_{i=1}^l q_i \alpha_i\in \Phi$ with $\bar\alpha=\bar\beta$
and $p_j=m_j$ for all $j\in I$, and let us check that also $q_j=m_j$ whenever $j\in I$. 
Fixed $j\in I$, then $\alpha_j\notin\Delta_0$ by i). 
If $\alpha_j$ is not joined by an arrow to any other node, then $0=\sum_{i=1}^l (q_i-p_i) \bar\alpha_i$ implies $q_j-p_j=0$. 
And, if $\alpha_j$ is joined to  $\alpha_k$ (which is necessarily unique, namely, $k=\mu(j)$ for $\mu$ the involution of the Dynkin diagram), 
the linear independence of $\{\bar\alpha_i:\alpha_i\in\Delta\setminus\Delta_0\}$
implies that $q_j-p_j+q_k-p_k=0$. 
We also have $p_k=m_k$, since ii) yields $k\in I$.
As $\tilde\alpha$ is the maximal root, $q_i\le m_i$ for any index $i$ and so $ m_j+m_k=p_j+p_k=q_j+q_k\le m_j+m_k$, forcing $q_j=m_j$.
For part c),  let us check that $\mathcal{B}_I= B_I\cap \mathfrak{g}$.  
On one hand, if  $\alpha \in  \Phi_I$, then
  \[
  \mathfrak{g}_{\bar\alpha}\stackrel{(1)}=(\oplus\{\mathfrak{g}^\CC_\gamma:\gamma\in\Phi,\bar\gamma=\bar\alpha\})\cap \mathfrak{g}\stackrel{b)}{=}(\oplus\{\mathfrak{g}^\CC_\gamma:\gamma\in\Phi_I,\bar\gamma=
  \bar\alpha\})\cap \mathfrak{g}\subseteq B_I \cap \mathfrak{g}.
  \]
This means that $\mathcal{B}_I\subseteq B_I\cap \mathfrak{g}$. 
On the other hand,  $(\mathfrak{g}^\CC)_{\gamma}\subseteq(\mathfrak{g}_{\bar\gamma})^\CC$ for any $\gamma\in\Phi$, so
  \[
  B_I=\oplus\{(\mathfrak{g}^\CC)_{\gamma}: \gamma\in\Phi_I\}\subseteq \oplus\{(\mathfrak{g}_{\bar\gamma})^\CC: \gamma\in\Phi_I\}\subseteq (\mathcal{B}_I)^\CC\subseteq (B_I \cap \mathfrak{g})^\CC\subseteq B_I,
  \]
  which implies that $(\mathcal{B}_I)^\CC= (B_I \cap \mathfrak{g})^\CC$ and $\mathcal{B}_I$ is a subset of $B_I\cap \mathfrak{g}$ with the same dimension, so necessarily equal. Finally 
  $
  [\mathcal{B}_I,[\mathcal{B}_I, \mathfrak{g}] ]\subseteq [B_I,[B_I,\mathfrak{g}]]\subseteq B_I
  $ and is also contained in $\mathfrak{g}$, so that it is contained in $\mathcal{B}_I$, showing that $\mathcal{B}_I$ is an inner ideal.
\end{proof}

Our objective is to prove that every abelian inner ideal of $\mathfrak{g}$ is conjugated to some $\mathcal{B}_I$ as in \cref{pr_BI}. First, we adapt \cite[Lemma~4.1]{Draper2012} in order to prove that, given any abelian inner ideal $B$ of $\g$, there is a decomposition as above  such that $B$ is homogeneous and $B\cap \mathfrak{g}_0=0$, that is, 
$B=\oplus_{\lambda\in\Sigma}B\cap \mathfrak{g}_\lambda$. Moreover, it is possible to prove something slightly stronger, namely:

\begin{lemma}\label{le_hom}
If $B$ is an abelian inner ideal of a simple real Lie algebra $\mathfrak{g}$, there is a maximal abelian subspace $\mathfrak{a}$ satisfying that $\kappa\vert_{\mathfrak{a}\times\mathfrak{a}}$ is a positive definite symmetric bilinear form, and that, if 
$\Gamma:\mathfrak{ g}=\mathfrak{g}_0\oplus\big(\hspace{-3pt}\oplus_{\lambda\in\Sigma}\mathfrak{g}_\lambda\big)$ is the decomposition in restricted root spaces 
$\mathfrak{ g}_{\lambda}=\{x\in\mathfrak{ g}:[h,x]=\lambda(h)x\quad\forall h\in\mathfrak{ a}\}$ as above, then $B$ is homogeneous and $\mathfrak{ g}_{\lambda}\subseteq B$ for each $\lambda\in\Sigma$ such that 
$\mathfrak{ g}_{\lambda}\cap B\ne0$. Besides, $\mathfrak{ g}_{0}\cap B=0$.
\end{lemma}

\begin{proof}
Take $0\ne e\in B$. By   \cref{le_lodeladescomposicion},  we can find an  $\mathfrak{sl}_2$-triple $\{e,h,f\}\subseteq\g$ such that 
$$
\mathfrak{g}=\mathfrak{g}_{-2}\oplus \mathfrak{g}_{-1}\oplus \mathfrak{g}_{0}\oplus \mathfrak{g}_{1}\oplus \mathfrak{g}_{2}
$$
is the diagonalization relative to $\ad h$, i.e., $\mathfrak{g}_{m}=\{x\in\g:[h,x]=m x\}   $. Of course $B$ is homogeneous for this grading, since 
$[h,B]=[[e,f],B]\subseteq [[e,B],f]+[e,[f,B]]=[e,[f,B]]\subseteq[B,[B,\mathfrak{g}]]\subseteq B$. Let us check that   $\mathfrak{g}_2\subseteq B$.
Indeed, for any $x\in\g_2$, $[h,x]=2x$. As $[e,x]=0$, the Jacobi identity tells that  $[e,[f,x]]=2x$. Now
$
[e,[f,[f,x]]]=[h,[f,x]]+[f,2x]=[f,2x]
$
and
$$
4x= [h,2x]=  [e,[f,2x]]= [e,[e,[f,[f,x]]]]\in [e,[e,\mathfrak{g}_{-2}]]\subseteq [B,[B,\g]]\subseteq B.
$$

If we take a second element $e'\in \mathfrak{g}_2$, again   there is an $\mathfrak{sl}_2$-triple $\{e',h',f'\}$ in $\g$. As $[e',[e',f']]=-2e'$, then \cite[Proposition~5.2]{Lopez2007}   (see also the proof of \cite[Lemma~4.1]{Draper2012}) says that we can replace $f'$ with another element  in $\mathfrak{g}_{-2}$. Thus  $h'=[e',f']\in \mathfrak{g}_0$ commutes with $h$, and we can consider the simultaneous diagonalization. 
In this way we can take a collection of elements $\{e_1,\dots,e_k\}\subseteq B\setminus\{0\}$ such that the related 
$\mathfrak{sl}_2$-triples $\{e_i,h_i,f_i\}_{i=1}^k$ satisfy $[h_i,h_j]=0$ for any $i,j$ and the $\ZZ^k$-induced grading on $\mathfrak{g}$,
$$
\Gamma_0: \mathfrak{g}=\oplus \mathfrak{g}_{(n_1,\dots,n_k)},\qquad \mathfrak{g}_{(n_1,\dots,n_k)}=\{x\in\mathfrak g:[h_i,x]=n_i x\quad \forall i\}
$$
 has  a maximal number of homogeneous components.  This grading satisfies 
 \begin{itemize}
 \item[a)] If $\mathfrak{g}_{(n_1,\dots,n_k)}\cap B\ne0$, then $\mathfrak{g}_{(n_1,\dots,n_k)}\subseteq B$;
 \item[b)] $\mathfrak{g}_{(0,\dots,0)}\cap B=0$;
 \item[c)] the restriction of the Killing form  $\kappa\vert_{\langle h_1,\dots,h_k\rangle}$ is positive definite.
 \end{itemize}
Indeed, if $0\ne e_{k+1}\in \mathfrak{g}_{(n_1,\dots,n_k)}\cap B$, we can take as before an $\mathfrak{sl}_2$-triple $\{e_{k+1},h_{k+1},f_{k+1}\} $ with $f_{k+1}\in \mathfrak{g}_{(-n_1,\dots,-n_k)}$. 
As the new grading is not a proper refinement of $\Gamma_0$, $e_{k+1}\in \mathfrak{g}_{(n_1,\dots,n_k)}=\mathfrak{g}_{(n_1,\dots,n_k)}\cap\{x\in\g:[h_{k+1},x]=2x\}$ which is contained in $B$ as above (the eigenspace of $\ad h_{k+1}$ of eigenvalue $2$ is). 
Also, if $\mathfrak{g}_{(0,\dots,0)}\subseteq B$, then $2e_1=[h_1,e_1]$ would belong to $[B,B]=0$, a contradiction.
Finally, $h=\sum_{i=1}^k\omega_ih_i$ satisfies $\kappa(h,h)=\sum_{(n_1,\dots,n_k)\in\ZZ^k}(\sum \omega_in_i)^2\dim\g_{(n_1,\dots,n_k)}\ge0$. And if  $\kappa(h,h)=0$, then $[h,\g]=0$ and $h\in Z(\g)=0$. 

Now we can take an abelian subspace $\mathfrak{a}$ containing $ \langle h_1,\dots,h_k\rangle$ maximal with the property  that  $\kappa\vert_{\mathfrak{a}\times\mathfrak{a}}$ is  positive definite. Let $\Gamma$ be the simultaneous diagonalization relative to $\{\ad h:h\in\mathfrak a\}$, which is a refinement of $\Gamma_0$, so that it still satisfies properties a) and b). (In particular this implies the homogeneity of $B$ for the grading $\Gamma$.)  
\end{proof}

Now, we can prove the converse of \cref{pr_BI} to describe, up to conjugation, all the abelian inner ideals of $\g$.

\begin{theorem}\label{mainth}
If $B$ is an abelian inner ideal of $\g$, there is $\emptyset\ne I\subseteq\{1,\dots,l\}$ adapted to the Satake diagram of $\g$ such that $B$ is conjugated to $\mathcal{B}_I$.
\end{theorem}

\begin{proof}
First we apply \cref{le_hom} to find a maximal abelian subspace $\mathfrak{a}$ with $\kappa\vert_{\mathfrak{a}\times\mathfrak{a}}$   positive definite  and such that the restricted root space  $\mathfrak{ g}_{\lambda}=\{x\in\mathfrak{ g}:[h,x]=\lambda(h)x\ \forall h\in\mathfrak{ a}\}$ is contained in $B$ for any $\lambda\in\Sigma$ such that  $\mathfrak{ g}_{\lambda}\cap B\ne0$.
Again consider $\mathfrak{h}$  any maximal abelian subalgebra of $\mathfrak{g}$ containing $\mathfrak{a}$.
Let $\Phi$  denote the root system of $\mathfrak{g}^\mathbb{C}$ relative to the Cartan subalgebra $\mathfrak{h}^\mathbb{C}$ and let $(\mathfrak{g}^\CC)_\alpha$ denote the complex one-dimensional root space for any $\alpha\in \Phi$. 
The root space decomposition of $\mathfrak{g}^\mathbb{C}$ is a grading $\tilde\Gamma$ which is a refinement of $\Gamma^\CC$, for $\Gamma:\mathfrak{ g}=\mathfrak{g}_0\oplus\sum_{\lambda\in\Sigma}\mathfrak{g}_\lambda$. 
Now note that $B^\CC$ is an abelian inner ideal of $\g^\CC$ homogeneous for $\tilde\Gamma$. 
A slight adaptation of  \cite[Theorem~4.4]{Draper2012} implies there are a basis $\Delta=\{\alpha_1,\dots,\alpha_l\}$ of $\Phi$ and a non-empty subset $I\subseteq\{1,\dots,l\}$ such that $B^\CC=B_{\mathfrak{h}^\CC,\Delta,I}= B_I$  as in \cref{eq_B_I}. 
To be precise, \cite[Theorem~4.4]{Draper2012} provides the Cartan subalgebra, while we have started with the Cartan subalgebra $\mathfrak{h}^\mathbb{C}$, but the proof in \cite{Draper2012} only uses that $B^\CC$ is homogeneous for $\tilde\Gamma$. Let us prove that 
 $$
 B=\mathcal{B}_I=\oplus\{\mathfrak{g}_{\bar\alpha}: \alpha\in\Phi_I\}.
 $$
Recall that $B=\oplus_{\lambda\in\Sigma}B\cap \mathfrak{g}_\lambda$, but $B\cap \mathfrak{g}_0= 0$ so there is $\Sigma'\subseteq \Sigma$ such that $B= \oplus_{\lambda\in\Sigma'} \mathfrak{g}_\lambda$. 
Our aim is to check that $\Sigma'=\Sigma^I$, where $\Sigma^I:=\{{\bar\alpha}: \alpha\in\Phi_I\}$.
First take $\lambda\in\Sigma'$ and $ \beta\in\Phi$ such that $\bar\beta=\lambda$. 
Since $\mathfrak{g}_\lambda=(\oplus_{\alpha\in\Phi}\{(\mathfrak{g}^\CC)_\alpha:\bar\alpha=\lambda\})\cap \mathfrak{g}\subseteq B$, we have
$$
(\mathfrak{g}^\CC)_\beta\subseteq (\mathfrak{g}_\lambda)^\CC\subseteq B^\CC=B_I=\oplus_{\alpha\in\Phi_I}(\mathfrak{g}^\CC)_\alpha .
$$
But $\dim_\CC (\mathfrak{g}^\CC)_\beta=1$, so that there is $\alpha\in\Phi_I$ such that $\alpha=\beta$ and in fact $\beta\in\Phi_I$. Hence not only $\Sigma'\subseteq\Sigma^I$, but we have proved that
\begin{equation}\label{eq_rem}
\textrm{if $ \beta\in\Phi$ such that $\bar\beta\in \Sigma'$ then $\beta\in\Phi_I$}.
\end{equation}
Second, take $\alpha\in\Phi_I$ and let us see that $\bar\alpha\in\Sigma'$. Again the one-dimensional space
$(\mathfrak{g}^\CC)_\alpha$ is contained in $B_I=B^\CC= \oplus_{\lambda\in\Sigma'} (\mathfrak{g}_\lambda)^\CC$, so that there is $\lambda\in\Sigma'\subseteq\mathfrak a^*$ such that $(\mathfrak{g}^\CC)_\alpha\subseteq (\mathfrak{g}_\lambda)^\CC=\oplus_{\beta\in\Phi}\{(\mathfrak{g}^\CC)_\beta:\bar\beta=\lambda\}$. 
Then $\bar\alpha=\lambda$, which belongs to $\Sigma'$. We have proved $\Sigma'=\Sigma^I$.
\smallskip

Note that the chosen subset $I$ may not be the required one, because it is not necessarily adapted to the Satake diagram of $\g$. But we can replace $I$ with a convenient subset:
Take $\tilde I\subseteq I$ minimal with the property that $\Phi_{\tilde I}=\Phi_{I}$. 
This can be done since $J\subseteq I$ implies $\Phi_{I}\subseteq \Phi_{J}$. 
The new $\tilde I$ satisfies i) in \cref{de_adapted}.
Indeed, assume there is $j\in\tilde I$ such that $\bar\alpha_j=0$. 
By minimality, $\Phi_{\tilde I\setminus\{j\}}\ne \Phi_{\tilde I}$. 
Then take $\alpha=\sum_{i=1}^lp_i\alpha_i$ with $p_j\ne m_j$ but $p_i=m_i$ for any $i\in \tilde I\setminus\{j\}$. 
Taking into account the combinatorial properties of the roots, there is a path connecting $\alpha$ with the maximal root $\tilde\alpha=\sum_{i=1}^lm_i\alpha_i$, that is, there are $\{i_1,\dots,i_s\}\subseteq \{1,\dots,l\}$ (not necessarily different indices) such that every element in the list
\begin{equation}\label{eq_list}
\alpha,\ \alpha+\alpha_{i_1}  ,\ \alpha+\alpha_{i_1} +\alpha_{i_2} ,\ \dots \ ,\alpha+\alpha_{i_1} +\dots+\alpha_{i_s} =\tilde \alpha
\end{equation}
 is a root. 
 Note that  $j\in \{i_1,\dots,i_s\} $ and that  every element in \eqref{eq_list} belongs to $\Phi_{\tilde I\setminus\{j\}}$. 
 Choose $\beta$ the last element in the list  \eqref{eq_list} such that $\beta+\alpha_j$ belongs to \eqref{eq_list} too. 
 Hence $\beta+\alpha_j\in\Phi_{\tilde I}$ and $\beta\notin\Phi_{\tilde I}$. 
 This means that $\overline{\beta+\alpha_j}\in\Sigma'$ but $\bar\beta\notin\Sigma'$ by \cref{eq_rem}, which is a contradiction since $\overline{\beta+\alpha_j}=\overline{\beta}+\overline{\alpha_j}=\bar\beta$. 
 This finishes the proof in case that the Satake diagram does not contain an arrow. 
 Otherwise, let $\mu$ be the order 2 automorphism of the Dynkin diagram fixing the Satake diagram, that is, $\overline{\alpha_{i}}= \overline{\mu(\alpha_{i})}$ for any $\alpha_i\in \Delta$.  
 (This only happens for some real forms of type $A$, $D$ or $E_6$.) 
 Think of $\mu$ as an automorphism of the set of indices $\{1,\dots,l\}$, i.e.,  $\alpha_{\mu(i)}=\mu(\alpha_{i})$. 
 Now replace $\tilde I$ with $\tilde I\cup \mu(\tilde I)$, which of course satisfies property ii) and it still satisfies i). 
 There is no problem with this replacement since $\Sigma'=\Sigma^{\tilde I}=\Sigma^{\tilde I\cup \mu(\tilde I)}$. 
 Indeed, choose $j\in  \tilde I$ (so that $\alpha_j$ corresponds to a white node) and let us see that $\Sigma^{\tilde I}=\Sigma^{\tilde I\cup \mu(j)}$. 
 Take $\lambda\in \Sigma^{\tilde I}$. 
 This means that $\lambda=\bar\alpha$ with $\alpha=\sum_{i=1}^lp_i\alpha_i\in\Phi$ such that $p_i=m_i$ for any $i\in\tilde I$ (in particular, $p_j=m_j$). 
 As $\overline{\mu(\alpha)}=\overline{\alpha}$ ($\mu$ acts $\ZZ$-linearly in $\Phi$ and fixes the Satake diagram), also $\lambda=\overline{\mu(\alpha)}$ and $ {\mu(\alpha)}\in \Phi_{\tilde I}$ by \cref{eq_rem}. 
 In particular the coefficient of the root $\alpha_j$ in $ {\mu(\alpha)}$ is $m_j$. 
 But $ {\mu(\alpha)}=p_1\mu(\alpha_1)+\dots +p_l\mu(\alpha_l)=p_1\alpha_{\mu(1)}+\dots+p_l\alpha_{\mu(l)}$ has as coefficient of the root $\alpha_j= \alpha_{\mu(\mu(j))}$ just $p_{\mu(j)}$, so that $p_{\mu(j)}=m_j=m_{\mu(j)}$ and then $\alpha\in \Phi_{\tilde I\cup \mu(j)}$.
\end{proof}

\begin{remark}\label{re_methodold}
Some comments are in order. 
First, note that $\{\overline{\alpha_i}:1\leq i\leq l\}\setminus\{0_{\mathfrak a^*}\}$ is a basis of the abstract root system $\Sigma$ (of course taking out the elements appearing twice), because every element in $\Sigma$ is the restriction of an integral linear combination of elements in $\Delta$ with all the coefficients having the same sign. 
The positive roots $\Sigma^+$ are the restrictions to $\mathfrak a$ of the positive  non-compact roots in $\Phi^+\setminus\Phi_0$, and $\tilde\beta:=\tilde\alpha\vert_{\mathfrak a}$ is the maximal root of  $\Sigma$ for this choice of basis. 
The only caution is that the basis of $\Sigma$ obtained in this natural way is not ordered in the usual way. 
In any case, let us denote by $\beta_i=\overline{\alpha_{i}} $ for each $\alpha_i\notin\Delta_0$ (also $\beta_i= \overline{\alpha_{\mu(i)}}$ if the Satake diagram has arrows). So the set of simple roots of $\Sigma$ is $\{\beta_i:i\in K\}$ for certain $K\subseteq\{1,\dots,l\}$ of cardinality $r=\dim\mathfrak a$. 
Then the maximal root $\tilde\beta=\sum_{i\in K}M_i\beta_i$ of $\Sigma$ has coefficients    $M_i=2m_i$ if the Satake diagram has arrows and $i\ne\mu(i)$, and 
 $M_i=m_i$ otherwise.
Note now that the set  related to $\mathcal B_I$ defined by $\Sigma^I=\{\bar\alpha:\alpha\in\Phi_I\}$, coincides with  
 $\Sigma_I=\{\sum_{i\in K} q_i\beta_i\in\Sigma:q_i=M_i\  \forall i\in I\}$ when the subset $I$  is adapted to the Satake diagram of $\g$, i.e., $I$ satisfies i) and ii).
 
Both viewpoints are useful for describing the inner ideals of a concrete simple Lie algebra $\g$. We will follow mainly that one in \cref{mainth},
  which allows  us to determine the dimensions of the inner ideals of $\g$ by counting roots in $\g^\CC$ or simply using the tables and lattices  in \cite{Draper2012}. This is quite easy since $\dim\mathcal B_I=| \Phi_I|$ (in general $\ne| \Sigma_I|$).
In the second approach, the advantage is that  the root system  $\Sigma$ is considerably smaller than $\Phi$, and the corresponding sets $\Sigma_I$ are equally listed  (at least for the reduced root systems).
But $\dim\mathcal B_I=\sum_{\lambda\in\Sigma_I}m_\lambda$,  so that  the knowledge of the restricted multiplicities is necessary.
An example is shown in \cref{reotrometodo}.  
 \end{remark}


\section{Classification of inner ideals of real simple Lie algebras}\label{section_Explicit classification}

Through this section, we will assume we have fixed a Cartan subalgebra of a simple split (real or complex) Lie algebra  and a set of simple roots of the related root system, labelled as in \cref{se_losdiagramasconcretos}.

Recall that it is possible that $\B_I=\B_J$ happens for some $I\neq J$. We will say that a subset $I\subseteq\{1,\dots,l\}$ is \emph{maximal} describing an inner ideal $B$ if $B=\B_I$ and whenever there is another $J\subseteq\{1,\dots,l\}$ with $B=\B_J$ then $J\subseteq I$. Sometimes it is useful to have this maximal $I$ since, only in that case we can assure that   $\B_I\subseteq \B_J$ implies  $J\subseteq I$ (recall that the converse was always true). We will use this to construct the lattices of inner ideals.  Similarly, we will say that a subset $I\subseteq\{1,\dots,l\}$ is \emph{minimal} describing an inner ideal $B$ if $B=\B_I$ and whenever there is another $J\subseteq\{1,\dots,l\}$ with $B=\B_J$ then $J\supseteq I$. We can always find both a maximal set and a minimal set describing any proper inner ideal. The minimal choice is useful in order to apply \cref{mainth}, since we  have only  to check whether or not the minimal set representing a determined inner  ideal  is  adapted to the Satake diagram.

Our description and classification of  the inner ideals will be up to conjugation, although we will not explicitly recall this every time. 
So ``the  non-zero abelian inner ideals are" should be read as ``up to conjugation, the  non-zero abelian inner ideals are".

\subsection{Type $A_l$}

Consider first the split Lie algebra, whose inner ideals are determined analogously to those ones of its complexified algebra. 
Recall that the set of positive roots is
$$
\Phi^+=\{\a_i+\a_{i+1}+\dots+\a_{j}:1\le i\le j\le l
\}
$$
and the maximal root is $\tilde\a=\a_{1}+\dots+\a_{l}$. 
The  non-zero abelian inner ideals of $\slf_{l+1}(\RR)$ are
$$
\{\B_{\{s,t\}}:1\le s\le t\le l\},
$$
because $\B_{I}=\B_{\{\textrm{min}\, I,\,\textrm{max} \,I\}}$. Here we enclose the possibility $\B_{\{s,s\}}\equiv \B_{\{s\}}$.
The corresponding roots are 
$$
\Phi_{\{s,t\}}=\{\a_{i}+\a_{i+1}+\dots+\a_{s}+\a_{s+1}+\dots+\a_{t}+\a_{t+1}+\dots+\a_{k}:1\leq i\leq s\leq t\leq k\leq l\},
$$
so that $\B_{\{s,t\}}$ has dimension $s(l+1-t)$. 
Since it is conjugated to $ \B_{\{l+1-t,l+1-s\}}$,   the set of  non-zero abelian inner ideals up to conjugation coincides with $\{\B_{\{s,t\}}:1\le s\leq t\leq l,s+t\leq l+1\}$. 
(This set has size $k^2$ if $l=2k-1$ and $k^2+k$ if $l=2k$.) 
The lattice is described  by taking into account that 
$$
\B_{\{s,t\}}\subseteq \B_{\{s',t'\}}\Leftrightarrow s\le s',t\ge t'.
$$
The only minimal abelian inner ideal (unique if we assume it adapted to the Cartan subalgebra and to the fixed ordering) is $\B_{\{1,l\}}=\g_{\tilde\a}$, while  the maximal ones are $\B_{\{k\}}$, with $k\in\{1,\dots,l\}$ (although $\B_{\{k\}}\cong \B_{\{l+1-k\}}$). For the other  non-compact real forms of $\slf_{l+1}(\CC)$, we use the results in the previous section. Thus,

\begin{proposition}
\begin{itemize}
\item There are $p$  non-zero abelian inner ideals of $\suf_{p,l+1-p}$, namely,
$$
\begin{cases}\B_{\{1,l\}}\subset\B_{\{2,l-1\}}\subset\dots\subset\B_{\{p,l+1-p\}}&\textrm{if }2p\leq l,\\
\B_{\{1,l\}}\subset\B_{\{2,l-1\}}\subset\dots\subset\B_{\{p-1,p+1\}}\subset\B_{\{p\}}&\textrm{if }2p=l+1.\end{cases}
$$
\item  If $l $ is odd,  the   non-zero abelian inner ideals of $\slf_m(\mathbb H)$ for $m=\frac{l+1}2$ are:
$$
\{\B_{\{2s,2t\}}:1\le s\le t\le m-1\}.
$$
\end{itemize}
\end{proposition}

\begin{proof}
Simply take into account that the Satake diagram of $\suf_{p,l+1-p}$ ($p\ge1$) has white nodes $\{1,\dots,p\}\cup\{l+1-p,\dots,l\}$ and  arrows connecting the $i$th node with the $(l+1-i)$th node for all $i\le p$; while $\slf_m(\mathbb H)$ has white nodes $\{2,4,\dots,l-1\}$ and no arrows. (The compact form is $\suf_{0,l+1}$, without abelian inner ideals, corresponding to $p=0$.)
\end{proof}

In particular, $\slf_m(\mathbb H)$ has no extremal points (more precisely, its minimal inner ideals have dimension 4) and the longest chain of proper inner   ideals has length $m-1$ (coinciding with the real rank of the algebra), for instance, $\B_{\{2,l-1\}}\subset \B_{\{4,l-1\}}\subset \dots \subset \B_{\{l-1,l-1\}}$.
In the list of the  abelian inner ideals of $\slf_m(\mathbb H)$  above, there are $\binom{m}2$    inner ideals, but again some of them are conjugated. 

\subsection{Type $B_l$, $l\ge2$.}  There are $2l-2$  non-zero abelian inner ideals of the split Lie algebra $\sof_{l,l+1}$ (as in the complex case). Recall that the sets of short and long positive roots are, respectively,
$$
\begin{array}{l}
\Phi_c^+=  \{ \a_{i}+\dots+\a_{l}:1\le i\le l  \}, \\
\Phi_l^+=  \{ \a_{i}+\dots+ \a_{j-1}, \a_{i}+\dots+\a_{j-1}+2\a_{j}+\dots+2 \a_{l} :1\le i<j\le l \}   .
\end{array}
$$ 
The maximal root is $\tilde\a=\a_{1}+2\a_{2}\dots+2\a_{l}$. The abelian inner ideals are 
\begin{equation}\label{eq_idsB}
\{\B_{\{k\}}  : 1\le k\le l\}\cup\{ \B_{\{1,k\}}:3\le k\le l\}  ,
\end{equation}
because $\B_I=\B_{\{2\}}=\g_{\tilde\a} $ if $2\in I$, $\B_I=\B_{\{\textrm{min}\, I \}} $ if $1\notin I$ and 
$\B_I=\B_{\{1,\textrm{min}\, (I\setminus\{1\}) \}} $ if $1\in I$ and $2\notin I$.  These are precisely the minimal choices for representing subsets.
The dimensions are $\dim \B_{\{1,k\}}=k-1$ if $k\ne 1$, $\dim \B_{\{1\}}=2l-1$ and
$\dim \B_{\{k\}}=\binom{k}2$ if $k\geq 2$, because, for $k\ne1$,
$$\begin{array}{l}
\Phi_{\{k\}}=\{\a_{i}+\dots+\a_{j-1}+2\a_{j}+\dots+2 \a_{l}:1\le i<j\le k \}=\Phi_{\{k,k+1,\dots,l\}},\vspace{2pt}\\
\Phi_{\{1\}}=\{\a_{1}+\dots+\a_{j-1}+2\a_{j}+\dots+2 \a_{l},\a_{1}+\dots+\a_{j},\a_1:1<j\le l \},\vspace{2pt}\\
\Phi_{\{1,k\}}=\{\a_{1}+\dots+\a_{j-1}+2\a_{j}+\dots+2 \a_{l}:1<j\le k \}=\Phi_{\{1,k,k+1,\dots,l\}}.
\end{array}
$$
The  other  non-compact real forms of $\sof_{2l+1}(\CC)$ have the following inner ideals.

\begin{proposition} 
\begin{itemize}
\item 
There is only one (up to conjugation) non-zero abelian inner ideal of $\sof_{1,2l}$, namely $\B_{\{1\}}$ (of dimension $2l-1$). In particular it has no extremal elements.
\item 
There are $2p-2$  non-zero abelian inner ideals of $\sof_{p,2l+1-p}$ if $2\le  p\le l$, namely,
$
\B_{\{2\}}    \subset \B_{\{1\}} $ if $p=2$,
and
$$
\{\B_{\{k\}} :1\le k\le p\}\cup \{\B_{\{1,k\}} :3\le k\le p\}  
$$
if $p\ge3$.
Thus a chain of maximal length of non-zero abelian inner ideals has length $p$ (coinciding with the real rank of the algebra)
$$
\B_{\{2\}} \subset \B_{\{1,3\}} \subset\B_{\{1,4\}} \subset\dots\subset \B_{\{1,p\}} \subset \B_{\{1\}} .
$$
\end{itemize}
\end{proposition}

\begin{proof}
The Satake diagram of $\sof_{p,2l+1-p}$ has white nodes just $\{1,\dots,p\}$.
\end{proof}


\subsection{Type $C_l$, $l\ge3$.} There are $l$  non-zero abelian inner ideals of the split Lie algebra $\spf_{2l}(\RR)$. Recall that the sets of short and long positive roots are, respectively,
$$
\begin{array}{l}
\Phi_c^+=  \{ \a_{i}+\dots+\a_{j-1},\a_{i}+\dots+\a_{j-1}+2\a_{j}+\dots+2\a_{l-1}+\a_{l}:i<j\le l  \}, \\
\Phi_l^+=  \{ 2\a_{i}+\dots+2\a_{l-1}+\a_{l} :1\le i\le l \}   .
\end{array}
$$ 
(The notation here is slightly confusing, for instance the long root for $i=l$ is $\a_l$.)
The maximal root is $\tilde\a=2\a_{1}+\dots+2\a_{l-1}+\a_{l}$. It is clear that $\B_I=\B_{\{\textrm{min}\, I \}} $, so that
the abelian inner ideals are just
$\{\B_{\{k\}}:1\le k\le l\}$, where $\B_{\{k\}}$ has dimension $\binom{k+1}{2}$, since
$$
\Phi_{\{k\}}=\{ 2\a_{i}+\dots+2\a_{l-1}+\a_{l} :1\le i\le k \} \cup\{\a_{i}+\dots+\a_{j-1}+2\a_{j}+\dots+2\a_{l-1}+\a_{l}:1\leq i<j\leq k \} .
$$ 
Thus we have a chain $\B_{\{1\}}=\g_{\tilde\a}\subset\B_{\{2\}}\subset\dots\subset\B_{\{l\}}$ of inner ideals of maximal length, $l$. In particular $\spf_{2l}(\RR)$ has extremal elements (but no inner ideals of dimension 2). Let us apply the results in the above section to the remaining  symplectic real algebras.

\begin{proposition}
There are $p$  non-zero abelian inner ideals of the   Lie algebra $\spf_{p,l-p}$ ($2p\le l$), which in fact form a (maximal) chain, namely,
$$
\B_{\{2\}}\subset\B_{\{4\}}\subset\dots\subset\B_{\{2p\}}.
$$
\end{proposition}

\begin{proof}
The Satake diagram of $\spf_{p,l-p}$ has white nodes $\{2,4,\dots,2p\}$ if $p\ne0$ (none if $p=0$). 
\end{proof}

The word \emph{maximal} here simply means of maximal length. In particular $\spf_{p,l-p}$ does not possess  extremal elements for any $p$, and $\spf_{1,l-1}$  only has one non-zero abelian inner ideal up to conjugation.

\subsection{Type $D_l$, $l\ge 4$.} 
\subsubsection{Type $D_l$, $l>4$.} 

There are $2l-2$  non-zero abelian inner ideals of the split Lie algebra $\sof_{l,l}$, with $l>4$. 
Recall that the set of positive roots is
\begin{align*}
\Phi^+=&\{\a_{i}+\dots+\a_{j-1}+2\a_{j}+ \dots+2\a_{l-2}+\a_{l-1}+\a_{l}:1\leq i<j\le l-2\}\cup\{\a_{l}\}\\
&\cup\{\a_{i}+\dots+\a_{j}:1\leq i\le j\le l-1\} \cup \{\a_{i}+\dots+\a_{l},\a_{i}+\dots+\a_{l-2}+\a_{l}:1\leq i\leq l-2\}  
\end{align*}
and the maximal root is $\tilde\a=\a_{1}+2\a_{2}+\dots+2\a_{l-2}+\a_{l-1}+\a_{l}$. 
A set of $2l-2$ representatives of the conjugacy classes of the  non-zero abelian inner ideals is 
\begin{equation}\label{eq_idsD}
\{\B_{\{k\}},\B_{\{1,k\}}:3\leq k\le l-1\}\cup \{\B_{\{1\}},\B_{\{2\}} ,\B_{\{l-1,l\}}, \B_{\{1,l-1,l\}} \}.
\end{equation}
To be precise, for $2\le k\le l-2$,
\begin{itemize}
\item $\Phi_{\{k\}} =\Phi_{\{k,\dots, l\}} = \{\a_{i}+\dots+\a_{j-1}+2\a_{j}+ \dots+2\a_{l-2}+\a_{l-1}+\a_{l}:1\leq i<j\le k\}$ has $\binom k2$ elements, (coinciding with $\g_{\tilde\a}$ for $k=2$),
\item $\Phi_{\{1,k\}} =\Phi_{\{1,k,\dots, l\}} = \{\a_{1}+\dots+\a_{j-1}+2\a_{j}+ \dots+2\a_{l-2}+\a_{l-1}+\a_{l}:1<j\le k\}$ has $  k-1$ elements,
\item $\Phi_{\{1\}} =\Phi_{\{1,l-2\}}\cup \{\a_1+\dots+\a_j:1\le j\le l\}\cup\{\a_1+\dots+\a_{l-2}+\a_l \}$ has $2l-2$ elements and it corresponds with a maximal inner ideal,
\item $\Phi_{\{1,l-1\}} =\Phi_{\{1,l-2\}}\cup \{\a_1+\dots+\a_{l-1},\a_1+\dots+\a_l\}$ has $l-1$ elements,
\item $\Phi_{\{1,l-1,l\}} =\Phi_{\{1,l-2\}}\cup \{\a_1+\dots+\a_l\}$ has $l-2$ elements,
\item $\Phi_{\{l-1,l\}} =\Phi_{\{l-2\}}\cup \{\a_i+\dots+\a_l:i\le l-2\}$ has $\binom{l-1}2$ elements, and
\item $\Phi_{\{l-1\}} =\Phi_{\{l-1,l\}}\cup \{\a_i+\dots+\a_{l-1}:i\le l-1\}$ has $\binom{l}2$ elements.
\end{itemize}

Here $\B_{\{1\}}$ and $\B_{\{l-1\}}\cong \B_{\{l\}}$ are the maximal inner ideals, while $\B_{\{1,2\}}\cong\B_{\{2\}}=\g_{\tilde\a}$ is the minimal one. Note that also $\B_{\{1,l-1\}}\cong \B_{\{1,l\}}$. The longest chain of proper inner ideals has length $l$. Note that, for $3\le k\le k'\le l-2$, we have 
$\B_{\{1,k\}}\subseteq \B_{\{1, k'\}}$, $ \B_{\{1,k\}}\subseteq \B_{\{ k'\}}$ and $\B_{\{k\}}\subseteq \B_{\{ k'\}}$.


\begin{proposition}\label{prop_tipoD}
\begin{itemize}
\item  There are $2p-2$  non-zero abelian inner ideals of $\sof_{p,2l-p}$ if $2\le  p\le l$  and $1$ if $p=1$, namely,
\begin{itemize}
\item  $\{\B_{\{1\}}\} $ if $p=1$;
\item  $\{\B_{\{1\}},\B_{\{2\}}\} $ if $p=2$;

\item  $\{\B_{\{k\}} :1\le k\le p\}\cup \{\B_{\{1,k\}} :3\le k\le p\}$ if $3\le  p\le l-2$;
\item   $\{\B_{\{k\}} :1\le k\le l-2\}\cup \{\B_{\{1,k\}} :3\le k\le l-2\}\cup \{\B_{\{l-1,l\}},\B_{\{1,l-1,l\}}\} $ if $p=l-1$.
\end{itemize}

\noindent
(If $p=0$ the algebra is compact and it does not contain any  non-zero abelian inner ideal, and, if $p=l$ the algebra is the above considered split algebra). 
\item  If $l $ is even (respectively odd), there are $\frac{l}2$ (respectively $\frac{l-1}2$)  non-zero abelian inner ideals of $\mathfrak{u}^*_l(\mathbb H)$, namely:
\begin{itemize}
\item  $\B_{\{2\}}\subset \B_{\{4\}}\subset\dots \subset\B_{\{l-2\}}\subset \B_{\{l\}} $ if $l$ is even;

\item   $\B_{\{2\}}\subset \B_{\{4\}}\subset\dots\subset \B_{\{l-3\}}\subset \B_{\{l-1,l\}} $ if $l$ is odd.
\end{itemize} 
\end{itemize}
\end{proposition}

\begin{proof}
The Satake diagram of $\sof_{p,2l-p}$ has 
\begin{itemize}
\item white nodes $\{1,\dots,p\}$ and no arrows, if $1\le p\le l-2$, 
\item white nodes $\{1,\dots,l\}$ and an arrow connecting the $l-1$th node with the $l$th node, if $p=l-1$; 
\end{itemize}
while the Satake diagram of  $\mathfrak{u}^*_l(\mathbb H)$ has 
\begin{itemize}
\item white nodes $\{2,4,\dots,l\}$ and no arrows, if $l$ is even; 
\item white nodes $\{2,4,\dots,l-1\}\cup\{l\}$ and an arrow connecting the $l-1$th node with the $l$th node, if $l$ is odd. 
\end{itemize}
\end{proof}
In particular, neither $\mathfrak{u}^*_l(\mathbb H)$ nor $\sof_{1,2l-1}$ has extremal elements.

\subsubsection{Type $D_4$.} 
As there are more automorphisms of the algebra of type $D_4$,   there are less inner ideals. To be precise, there are 4 inner ideals up to conjugation:
$$
\B_{\{1,2,3,4\}}\subset \B_{\{1,3,4\}}\subset \B_{\{1,4\}}\subset \B_{\{1\}},
$$
of dimensions $1$, $2$, $3$ and $6$ respectively.  (The map $\alpha_1\mapsto \alpha_3\mapsto \alpha_4\mapsto \alpha_1$ fixing the $2^{\mathrm{nd}}$-node is a diagram automorphism.) The other  non-compact real forms work as in \cref{prop_tipoD}, namely, $\sof_{p,8-p}$ with a chain of maximal length $p$ ($p=0$ compact, and $p=4$ split):
\begin{itemize}
\item Only $
  \B_{\{1\}}
$ up to conjugation for $\sof_{1,7}$, an inner ideal of dimension 6;
\item $
\B_{\{1,2,3,4\}} \subset \B_{\{1\}}
$ for $\sof_{2,6}$, inner ideals of dimensions $1$ and $6$;
\item 
$
\B_{\{1,2,3,4\}}\subset \B_{\{1,3,4\}}\subset   \B_{\{1\}}
$ for $\sof_{3,5}$, inner ideals of dimensions $1$, $2$  and $6$.
\end{itemize}
\subsection{Type $E_6$.} 
The maximal root is $\a_1+2\a_2+2\a_3+3\a_4+2\a_5+\a_6$.
There are $7$  non-zero abelian inner ideals of the split algebra $\mathfrak{e}_{6,6}$, given by
\begin{equation}\label{eq_idsE6}
\begin{aligned}
\B_{\{ 1,2,3,4,5,6 \}}=\B_{\{ 2 \}}=\g_{\tilde\a},\qquad 
&\B_{\{ 1,3,4,5,6 \}}=\B_{\{ 4 \}},\qquad 
&&\B_{\{ 1,3,5,6 \}}=\B_{\{  3,5\}},\\
\B_{\{ 1,3,6 \}}=\B_{\{ 3,6 \}}\cong \B_{\{ 1,5 \}},\qquad 
&\B_{\{ 1,3 \}}=\B_{\{  3\}}\cong \B_{\{ 5 \}},\qquad 
&&\B_{\{1,6  \}} ,\qquad 
\B_{\{ 1 \}}\cong \B_{\{ 6 \}} ;
\end{aligned}
\end{equation} 
with dimensions $1$, $2$, $3$, $4$, $5$, $8$ and $16$, respectively. Here we have to take care with the order 2 automorphism of the Dynkin diagram. 
We have chosen, for each inner ideal $\B$, the set $I\subseteq\{1,\dots,6\}$ to the left (right, respectively)   maximal (minimal, respectively)  among the ones satisfying $\B=\B_I$. 
Recall that the choice of a maximal $I$ helps us to find a chain of inner ideals of maximal length, while a minimal $I$ gives us 
  our main tool to know if a concrete real form possesses such inner ideal: it reduces  to check if $I$ is or is not adapted to the Satake diagram.

\begin{proposition}\label{pr_idsE6}
\begin{itemize}
\item  There are $4$  non-zero abelian inner ideals of the   Lie algebra $\mathfrak{e}_{6,2}$, of dimensions $1$, $2$, $3$ and $8$, namely, $\B_{\{2\}}\subset \B_{\{4\}}\subset \B_{\{3,5\}}\subset  \B_{\{1,6\}}$.
\item There are $2$  non-zero abelian inner ideals of the Lie algebra $\mathfrak{e}_{6,-14}$, of dimensions $1$ and $8$, namely, $\B_{\{2\}}\subset \B_{\{1,6\}}$.
\item There are $2$  non-zero abelian inner ideals of the Lie algebra $\mathfrak{e}_{6,-26}$, of dimensions $8$ and $16$, namely,   $\B_{\{1,6\}}\subset \B_{\{1\}}$.
\end{itemize}
\end{proposition}

\begin{proof}
The Satake diagram of the real form $\g$ has 
\begin{itemize}
\item white nodes $\{1,2,3,4,5,6 \}$ and arrows connecting the nodes $3$ and $5$ and  the nodes $1$ and $6$, if $\g=\mathfrak{e}_{6,2}$;
\item white nodes $\{1,2,6 \}$ and arrows connecting the nodes $1$ and $6$, if $\g=\mathfrak{e}_{6,-14}$;
\item white nodes $\{1,6 \}$ and no  arrows, if $\g=\mathfrak{e}_{6,-26}$.
\end{itemize}
\end{proof}
Hence  $ \mathfrak{e}_{6,-26}$ has no pure extremal elements (there are no $\mathbb Z$-gradings with a corner of dimension 1).

\begin{remark} \label{reotrometodo}
According to \cref{re_methodold}, there is an alternative method to arrive at the same results. Note that $\Phi\setminus \Phi_0$ is a root system of type $F_4$, $BC_2$ and $A_2$ if the signature of the real form is $2$, $-14$ or $-26$ respectively. For instance, in the last case, the inner ideals of a split    algebra of type $A_2$ are the root spaces $B^1=\g_{\a'_1}$ and $B^2=\g_{\a'_1}\oplus\g_{\a'_1+\a'_2}$, where $\a'_1=\a_1\vert_{\mathfrak{a}}$ and $\a'_2=\a_6\vert_{\mathfrak{a}}$. It is well known that 
the restricted multiplicities are $m_{\bar\a_1}=8=m_{\bar\a_6}$ \cite{Helgason2001} (see also \cite{DraperFontanals2016}, with explicit constructions of the Satake diagrams of these real forms), so that $\dim B^1=8$ and $\dim B^2=16$. Similarly the structure of the inner ideals of the complex Lie algebras of type $F_4$ and $B_2$ tells us that for $\mathfrak{e}_{6,2}$ there will be 4 abelian inner ideals, and for $\mathfrak{e}_{6,-14}$ there will be only 2.  In order to compute the explicit dimensions we again need the restricted multiplicities. 
\end{remark} 

\subsection{Type $E_7$.} \label{se_tipoe7}
The maximal root is $2\a_1+2\a_2+3\a_3+4\a_4+3\a_5+2\a_6+\a_7$.
There are $10$  non-zero abelian inner ideals of the split algebra $\mathfrak{e}_{7,7}$, given by
\begin{equation*} 
\begin{aligned}
\B_{\{ 1,2,3,4,5,6,7 \}}=\B_{\{ 1  \}},\quad 
&\B_{\{ 2,3,4,5,6,7 \}}=\B_{\{  3 \}},\quad 
&&\B_{\{ 2,4,5,6,7 \}}=\B_{\{ 4  \}},\\
\B_{\{ 2,5,6,7 \}}=\B_{\{ 2,5  \}},\quad 
&\B_{\{ 2,6,7 \}}=\B_{\{ 2,6  \}},\quad 
&&\B_{\{5,6,7  \}}=\B_{\{ 5 \}} ,\\ 
\B_{\{ 2,7 \}}  , \quad
&\B_{\{ 6,7 \}}=\B_{\{ 6 \}} , \quad
&&\B_{\{ 2 \}}, \quad
\B_{\{ 7 \}} ;
\end{aligned} 
\end{equation*}
of dimensions $1$, $2$, $3$, $4$, $5$, $5$, $6$, $10$, $7$ and $27$, respectively. 

\begin{proposition}
\begin{itemize}
\item  There are $4$  non-zero abelian inner ideals of the Lie algebra $\mathfrak{e}_{7,5}$, of dimensions $1$, $2$, $3$ and $10$, namely, $\B_{\{1 \}}\subset \B_{\{3 \}}\subset \B_{\{4 \}} \subset \B_{\{6 \}}$.
\item There are $3$  non-zero abelian inner ideals of the Lie algebra $\mathfrak{e}_{7,-25}$, of dimensions $1$, $10$ and $27$, namely,  $\B_{\{1 \}}\subset \B_{\{6 \}} \subset \B_{\{7\}}$.
\end{itemize}
\end{proposition}

\begin{proof}
The Satake diagram of $L$ has 
white nodes $\{1,3,4,6 \}$   if $L=\mathfrak{e}_{7,5}$;
and white nodes $\{1,6,7 \}$   if $L=\mathfrak{e}_{7,-25}$.
\end{proof}

\subsection{Type $E_8$.} \label{sub_e8}
The maximal root is $2\a_1+3\a_2+4\a_3+6\a_4+5\a_5+4\a_6+3\a_7+2\a_8$.
There are $10$  non-zero abelian inner ideals of the split algebra $\mathfrak{e}_{8,8}$, given by
\begin{equation*}\label{eq_idsE}
\begin{aligned}
\B_{\{ 1,2,3,4,5,6,7,8 \}}=\B_{\{ 8  \}},\quad 
&\B_{\{ 1,2,3,4,5,6,7 \}}=\B_{\{ 7  \}},\quad 
&&\B_{\{ 1,2,3,4,5,6 \}}=\B_{\{ 6  \}},\\
\B_{\{ 1,2,3,4,5 \}}=\B_{\{ 5  \}},\quad 
&\B_{\{ 1,2,3,4 \}}=\B_{\{  4 \}},\quad 
&&\B_{\{1,2,3  \}}=\B_{\{  2,3 \}} ,\\ 
\B_{\{ 1,2 \}}, \quad
&\B_{\{ 1,3 \}}=\B_{\{ 3 \}}, \quad
&&\B_{\{ 2 \}}, \quad
\B_{\{ 1 \}} ;
\end{aligned}
\end{equation*} 
with dimensions $1$, $2$, $3$, $4$, $5$, $6$, $7$, $7$, $8$ and $14$. 
Again we have chosen, for each inner ideal, the set $I$ to the left (right, respectively) maximal (minimal, respectively) among the ones representing the inner ideal.

\begin{proposition}
 There are $4$  non-zero abelian inner ideals of the Lie algebra $\mathfrak{e}_{8,-24}$, of dimensions $1$, $2$, $3$ and $14$, namely, $\B_{\{8\}} \subset \B_{\{7\}} \subset \B_{\{6\}}\subset \B_{\{1\}}$. 
\end{proposition}   

\begin{proof}
The Satake diagram of $\mathfrak{e}_{8,-24}$ has white nodes $\{1,6,7,8 \}$.
\end{proof}

\subsection{Type $F_4$.} 
The maximal root is $2\a_1+3\a_2+4\a_3+2\a_4$.
There are $4$  non-zero abelian inner ideals of the split algebra $\mathfrak{f}_{4,4}$, given by
\begin{align*}
\B_{\{ 1,2,3,4 \}} =\B_{\{ 1 \}} ,\quad
&\B_{\{ 2,3,4 \}} =\B_{\{ 2 \}} ,\\
\B_{\{  3,4\}} =\B_{\{ 3 \}} ,\quad
&\B_{\{ 4 \}}, \end{align*}
of dimensions $1$, $2$, $3$ and $7$ respectively. (Again we use the same convention for the  maximal and minimal  subsets.)

\begin{proposition}
 There is only  one  non-zero abelian inner ideal of the Lie algebra $\mathfrak{f}_{4,-20}$, of dimension $7$, namely, $\B_{\{4\}}$.
 That is, all the proper inner ideals are conjugated.
\end{proposition}

\begin{proof}
The Satake diagram of $\mathfrak{f}_{4,-20}$ has only one
 white node,  $\{4\}$.
\end{proof}

\subsection{Type $G_2$.} 
There are $2$  non-zero abelian inner ideals of the split algebra $\mathfrak{g}_{2,2}$, given by $\B_{\{1\}}$ and $\B_{\{2\}}$, of dimensions $1$ and $2$ respectively. Of course, the compact algebra $\mathfrak{g}_{2,-14}$ has no  non-zero abelian inner ideals.


\section{More on inner ideals of real exceptional Lie algebras}\label{se_moreonexceptional}

We consider our combinatorial description of the inner ideals insufficient. So, our next objective is, in the case of exceptional Lie algebras, to provide concrete descriptions and realizations of all these inner ideals.
 

\subsection{Preliminaries on structurable algebras}\label{se_estructurables}   

It is well known that any Jordan algebra allows to define  a 3-graded Lie algebra by means of the Tits-Koecher construction. Structurable algebras are a class of non-associative algebras  with involution introduced in \cite{Allison78}, containing the class of Jordan algebras, which also permits to construct  a $\mathbb Z$-graded Lie algebra  starting from any algebra in this class, in general with 5 pieces, by means of a generalized \emph{Tits-Kantor-Koecher construction}. Moreover, \cite[Theorem~10]{Allison1979} also proves that every finite-dimensional central simple Lie algebra over a field of characteristic zero containing some  non-zero ad-nilpotent element (for instance, with proper inner ideals) is obtained from a (central simple) structurable algebra by using this construction. This is the case of all the  non-compact real (finite-dimensional) simple Lie algebras.
 
 A recent work  by De Medts and Meulewaeter  \cite{Medts2020}  studies the set of inner ideals of the Lie algebras constructed from some  structurable algebras. With these inner ideals they construct Moufang sets, Moufang triangles and Moufang hexagons.

\begin{definition}
Let $(\mathcal A,-)$ be an algebra with involution (an order $2$ anti-homomorphism). For each $x,y\in\mathcal A$, denote by $V_{x,y}\colon\mathcal A\to\mathcal A$ the linear operator given by
$$
V_{x,y}(z)=(x\bar y)z+(z\bar y)x-(z\bar x)y,
$$
for any $z\in\mathcal A$. The algebra $\mathcal A$ is called a \emph{structurable algebra} if, for any $x,y,z,w\in\mathcal A$,
$$
[V_{x,y},V_{z,w}]=V_{V_{x,y}(z),w}-V_{z,V_{y,x}(w)}.
$$
In particular  $\textrm{Instr}(\mathcal A):=\{\sum V_{x_i,y_i}:x_i,y_i\in\mathcal A\}\equiv V_{\mathcal A,\mathcal A}$  is a Lie algebra.
\end{definition}

\begin{example}\label{ex_structurable} 
Some examples, extracted from \cite[\S8]{Allison78}, are:
\begin{itemize}
\item[i)]
If $(\mathcal A,-)$ is a unital associative algebra with involution, then it is structurable. 
\item[ii)] If $J$ is a unital Jordan algebra, that is, commutative and satisfying the Jordan identity ($(x^2y)x=x^2(yx)$), then $J$ is structurable for the involution given by the identity. %
\item[iii)] If $(C,n)$ is a unital composition algebra, that is, $C$ is endowed with a non-degenerate multiplicative quadratic form $n\colon C\to \mathbb F$ (that is, $n(xy)=n(x)n(y)$), then   $C$ has an involution $-$ such that $n(x)=x\bar x$ and it turns to be an structurable algebra. Moreover, if $C_1$ and $C_2$ are two unital composition algebras, then the tensor product $C_1\otimes C_2$ is structurable for the involution given by $\overline{x_1\otimes x_2}=\overline{x_1}\otimes\overline{x_2}$.
\end{itemize}
\end{example}

\begin{definition}
If $(\mathcal A,-)$ is an structurable algebra, we consider the sets of \emph{skew-symmetric} elements and of \emph{Hermitian} elements defined respectively by
$$
\mathcal S=\mathrm{Skew}(\mathcal A,-):=\{s\in\mathcal A:\bar s=-s\},\qquad\mathcal H:=\{h\in\mathcal A :\bar h=h\}.
$$
Clearly $\mathcal A=\mathcal H\oplus\mathcal  S$. The dimension of $\mathcal S$ is called the \emph{skew-dimension} of $\mathcal A$.
\end{definition}

As mentioned, the interesting point on structurable algebras is that they provide a 5-graded Lie algebra by means of a construction which generalizes the Tits-Koecher construction for Jordan algebras. And conversely, the isotropic finite-dimensional simple Lie  algebras over fields of characteristic zero can be constructed by the TKK-construction (abbreviation of  Tits, Kantor and Koecher) applied to some structurable algebra:   

\begin{definition}  (\cite[\S3]{Allison1979})
Let $(\mathcal A,-)$ be an structurable algebra with $\mathcal S=\mathrm{Skew}(\mathcal A,-)$ and consider  the $\mathbb{Z}$-graded vector space $\mathcal{K}(\mathcal A)=\mathcal{K}_{-2}\oplus\mathcal{K}_{-1}\oplus\mathcal{K}_{0}\oplus\mathcal{K}_{1}\oplus\mathcal{K}_{2}$,
for
$$
\begin{array}{ll}
\mathcal{K}_{2}:=\{(0,s):s\in \mathcal S\},\quad&\mathcal{K}_{1}:= \{(a,0):a\in\mathcal  A\},
\\
\mathcal{K}_{-2}:=\{(0,s)\tilde{\ }:s\in \mathcal S\},\qquad&\mathcal{K}_{-1}:= \{(a,0)\tilde{\ }:a\in\mathcal  A\},
\end{array}
\qquad \mathcal{K}_{0}:=\mathrm{Instr}(\mathcal  A,-)=V_{\mathcal  A,\mathcal   A},
$$
where $ {\tilde{\ }}$ simply denotes a copy of the pair. Then $\mathcal  K(\mathcal   A)$ is a (graded) Lie algebra for the product such that  $\mathrm{Instr}(\mathcal  A)$ is a Lie subalgebra and the following conditions hold:
\begin{equation}\label{eq_productoenKantor}
\begin{array}{l}
\star\quad {[}(a,r),(b,s)]=(0,a\bar b-b\bar a),\vspace{2pt}\\
\star\quad{[}(a,r)\tilde{\ } ,(b,s)\tilde{\ } ]=(0,a\bar b-b\bar a)\tilde{\ },\vspace{2pt} \\

\star\quad{[}(a,r)  ,(b,s)\tilde{\ } ]=(-sa,0)\tilde{\ } +V_{a,b}+L_rL_s+(rb,0),\vspace{2pt}\\
\star\quad{[}T,(a,r) ]=(Ta, T^\delta r),\vspace{2pt}\\
\star\quad{[}T,(a,r)\tilde{\ } ]=(T^\epsilon a, (T^\epsilon)^\delta r)\tilde{\ },\vspace{2pt}\\
 \end{array} 
\end{equation}
for any $T\in \mathrm{Instr}(\mathcal A,-)$, $a,b\in \mathcal{A}$, $r,s\in \mathcal{S}$,
where $L_s,R_s\colon \mathcal A\to \mathcal A$ denote the left and right multiplication operators by $s\in \mathcal S$,
 $T^\epsilon:=T-L_{T(1)+\overline{T(1)}}$ and $T^\delta:=T+R_{\overline{T(1)}}$.
This Lie algebra will be    called the \emph{Kantor construction} or  \emph{TKK-construction}  attached to the structurable algebra $\mathcal A$.
\end{definition}

In the complex case every simple Lie algebra is obtained in this way, while in the real case
 every simple non-compact Lie algebra is obtained in this way.

\subsection{Inner ideals of  $\mathcal K(\mathcal{A})$  for $\mathcal{A} $  tensor product of composition algebras}\label{sec_tensorpro}

It is clear that   $\mathcal K_2$ is an inner ideal of $\mathcal K(\mathcal A)$ for any structurable algebra $\mathcal A$. The purpose now is to collect 
   information about  the remaining inner ideals of  the Lie algebras $\mathcal  K(\mathcal   A)$ when  $\mathcal A$ is a tensor product of composition algebras.
   We are interested in them because  all the exceptional complex Lie algebras not of type $G_2$ arise  from the  case iii) in \cref{ex_structurable}:
  If $C_1$ and $C_2$ are two unital composition algebras over the complex numbers, with $C_2$ of dimension 8 (usually called a \emph{Cayley} algebra) then $\mathcal K(C_1\otimes C_2)$ is the exceptional algebra of type $F_4$, $E_6$, $E_7$ and $E_8$ according to the dimension of $C_1$ being $1$, $2$, $4$ and $8$, respectively. Hence, if $(C_1,n_1)$ and $(C_2,n_2)$ are two unital composition algebras over the real numbers, with $C_2$ a Cayley algebra, thus $\mathcal K(C_1\otimes C_2)$ is   a real form of an exceptional Lie algebra, $\mathcal K(C_1^\CC\otimes C_2^\CC)$. 
  In these cases, we will describe some inner ideals, and we will use the provided information to recognize the concrete real form obtained. 

  Note that the set of skew-symmetric elements of $\mathcal A=C_1\otimes C_2$, for $C_1$ and $C_2$ arbitrary unital real composition algebras, coincides with  $\mathcal S=(\mathcal S_1\otimes \RR)\oplus(\RR\otimes \mathcal  S_2)$, where $\mathcal S_i=\{x\in C_i:\bar x=-x\}$ denotes the set of skew-symmetric elements of the corresponding composition algebra. 
  A key tool to study this type of structurable algebras is to consider the so called \emph{Albert form} defined in \cite[\S3]{Allison1988} by
$$
q_{\mathcal A}\colon \mathcal S\to \RR,\qquad q_{\mathcal A}(s_1\otimes 1+1\otimes s_2)=n_1(s_1)-n_2(s_2).
$$
This allows us to find some of the inner ideals of $\mathcal K(\mathcal A)$, besides $\mathcal K_2$, which of course it is always an inner ideal that can be identified with $\mathcal S$.

\begin{proposition}\label{pr_Jeroen1}
The only inner ideals of  $ \mathcal K(C_1\otimes C_2)$ contained in the component $\mathcal K_2$ are:
\begin{itemize}
\item The whole $\mathcal K_2=\{(0,s):s\in \mathcal S\};$
\item $\{(0,s):s\in I\} $ for any subspace $I$  of $\mathcal S$ such that $q_{\mathcal A}(I)=0$.
\end{itemize}
In consequence, $\mathcal K(C_1\otimes C_2)$ has inner ideals of each dimension $1,\dots, m$ for $m$ the Witt index of $q_{\mathcal A}$ (the dimension of the maximal isotropic subspace), and also an inner ideal of dimension equal to $\dim \mathcal S=\dim C_1+\dim C_2-2$.
\end{proposition}

\begin{proof} Of course $ \mathcal K_2$ is an inner ideal. Now 
take $B=\{(0,s):s\in I\}\subseteq \mathcal K_2$ for a subspace $I$ of $\mathcal S$ such that $q_{\mathcal A}(I)=0$, and let us check that $[B,[B,\mathcal K]]\subseteq B$, or equivalently, due to the grading, that $[B,[B,\mathcal K_{-2}]]\subseteq B$. That is, we have to see, for any $s,s'\in I$ and $t\in \mathcal S$, if 
$[(0,s),[(0,s'),(0,t)\tilde{\ }]]=(0,-(L_{s'}L_t)^\delta s)=-(0,s'(ts)+s(ts'))$ belongs or not to $B$. By using \cite[Equation~3.7]{Allison1988}, we have that $s(ts)\in\langle s\rangle$ for any $s,t\in\mathcal S$ such that $q_{\mathcal A}(s)=0$. 
In our case $s,s',s+s'\in I$; so that $q_{\mathcal A}(s)=q_{\mathcal A}(s')=q_{\mathcal A}(s+s')=0$, so 
$-(s'(ts)+s(ts'))=s(ts)+s'(ts')-(s+s')(t(s+s'))\in\langle s,s'\rangle\subseteq I$, as we needed.
 
Conversely, if $B\subseteq \mathcal K_2$ is an inner ideal, there is a vector subspace $I\le\mathcal S$ with $B=\{(0,s):s\in I\}$. Let us see that, if there exists $s\in I$ with $q_{\mathcal A}(s)\ne0$, then $I$ coincides with the whole $\mathcal S$. Take into account that $[(0,s),[(0,s),(0,t)\tilde{\ }]]\in [B,[B,\mathcal K]]\subseteq B$ for any   $t\in \mathcal S$, which means that $s(ts)\in I$. Again \cite[Equation~3.7]{Allison1988} says that $q_{\mathcal A}(s)t^\natural=s(ts)+q_{\mathcal A}(s,t^\natural)s\in I,$ where
$(s_1\otimes 1+1\otimes s_2)^\natural:=s_1\otimes 1-1\otimes s_2$. Hence $t^\natural\in I$ for any $t\in\mathcal S$ and $I=\mathcal S$.
\end{proof}

Another inner ideal of the Lie algebra $\mathcal K(C_1\otimes C_2)$, which is not contained in $\mathcal K_2$, appears in case $C_2=\mathbb R\oplus \mathbb R$. This is a composition algebra with product componentwise and exchange involution. So $1_{C_2}=(1,1)=e_1+e_2$ for $e_1=(1,0)$ and $e_2=(0,1)$ orthogonal idempotents with $\overline{e_1}=e_2$ and $\overline{e_2}=e_1$. Note that $C_2e_i\subseteq \RR e_i$ for any $i=1,2$, and $e_i\bar e_i=0$. Besides, it is easy to prove that $C_1\otimes C_2\cong C_1\oplus C_1^{op}$ with $\overline{(x,y)}=(y,x)$.

\begin{proposition}\label{pr_Jeroen2} (See \cite[Lemma~6.10]{Medts2020})
If $C_1$ is a composition algebra and $C_2=\mathbb R\oplus \mathbb R$, then
$$
B=\{(a,0):a\in C_1\otimes (1,0)\}\oplus \{(0,s):s\in \mathcal S= (\mathcal S_1\otimes 1)\,\oplus\,( \mathbb R1\otimes(1,-1))\}
$$
is an inner ideal of  $ \mathcal K(C_1\otimes C_2)$ containing the component $\mathcal K_2$ of  dimension equal to $2\dim \mathcal S=2\dim C_1$.
\end{proposition}

\begin{proof}
Note that $B=B_1\oplus \mathcal K_2$ with $B_1=\{(a,0):a\in C_1\otimes e_1\}\subseteq \mathcal K_1$. Taking into account the $\mathbb Z$-grading on the Lie algebra, then checking $[B,[B,\mathcal K(C_1\otimes C_2)]]\subseteq B$ is equivalent to checking the following conditions:
\begin{itemize}
\item[a)] $[B_1,[\mathcal K_{ 2},\mathcal K_{-2}]]\subseteq B_1$ and $[\mathcal K_{ 2},[B_1,\mathcal K_{-2}]]\subseteq B_1$,  
\item[b)] $[B_1,[B_1,\mathcal K_{-2}]]=0$ and 
\item[c)] $[B_1,[B_1,\mathcal K_{-1}]]\subseteq B_1$.
\end{itemize}
Take $a,b\in C_1\otimes e_1$, $c\in C_1\otimes C_2$ and $s,t\in\mathcal S$. For item a),
$[(a,0),[(0,s),(0,t)\tilde{\ }]]=(-s(ta),0)\in I_1$ since $(C_1\otimes C_2)(C_1\otimes e_1)\subseteq C_1\otimes e_1$. The other expression in a) holds by applying the Jacobi identity. For b), let us see that
$[(a,0),[(b,0),(0,t)\tilde{\ }]]=-V_{a,tb}$ is the zero map. It is enough to see that $V_{a,b}(c)=0$, since $c\in\mathcal A$ is arbitrary and $tb\in C_1\otimes e_1$.
 Write $a=a_1\otimes e_1$, $b=b_1\otimes e_1$ and $c=c_1\otimes c_2$ (the argument goes equally if $c=\sum_i c^i_1\otimes c^i_2$). Then
$V_{a,b}(c)=(a\bar b)c+(c\bar b)a-(c\bar a)b 
=0$ since $a\bar b=0$, $(c\bar b)a=(c_1\bar b_1)a_1\otimes (c_2e_2)e_1=0$ and similarly changing $a$ and $b$.
For c), we compute
$[(a,0),[(b,0),(c,0)\tilde{\ }]]=-(V_{b,c}(a),0)$ but $V_{b,c}(a)=(b\bar c)a+(a\bar c)b-0\in  C_1\otimes C_2e_1\subseteq C_1\otimes e_1$.
\end{proof}

Note that the above two propositions are also valid for a field $\mathbb F$ and $C_2=\mathbb F\oplus\mathbb F$. But in the real case we get some useful conclusions, taking advantage of our descriptions  of the inner ideals in \cref{section_Explicit classification}. In particular, all the non-compact real forms of $E$-type appear:

\begin{corollary}\label{corolarioexc}
 Denote by $\mathbb{O}$ and $\mathbb{O}_s$ the octonion division algebra and the split octonion algebra, respectively. Denote by $\mathbb{C}_s=\mathbb{R}\oplus \mathbb{R}$.
 Then the non-compact real exceptional Lie algebras of $E_6$-type can be obtained as follows:
 $$
 \mathcal K(\mathbb{O}_s\otimes\mathbb{C}_s)\cong\mathfrak{e}_{6,6}\,;\quad
  \mathcal K(\mathbb{O}_s\otimes\mathbb{C})\cong\mathfrak{e}_{6,2}\,;\quad
   \mathcal K(\mathbb{O}\otimes\mathbb{C})\cong\mathfrak{e}_{6,-14}\,;\quad
    \mathcal K(\mathbb{O}\otimes\mathbb{C}_s)\cong\mathfrak{e}_{6,-26}.
     $$ 
     In the same way, denoting by   $\mathbb{H}$ and $\mathbb{H}_s\cong\mathrm{Mat}_{2\times 2}(\mathbb R)$ the quaternion division algebra and the split quaternion algebra, respectively, then the  non-compact real forms of $E_7$-type are 
     $$
 \mathcal K(\mathbb{O}_s\otimes\mathbb{H}_s)\cong\mathfrak{e}_{7,7}\,;\quad
  \mathcal K(\mathbb{O}_s\otimes\mathbb{H})\cong\mathfrak{e}_{7,5}\cong
   \mathcal K(\mathbb{O}\otimes\mathbb{H}) \,;\quad
    \mathcal K(\mathbb{O}\otimes\mathbb{H}_s)\cong\mathfrak{e}_{7,-25};
     $$ 
     while the  non-compact real forms of $E_8$-type are 
         $$
 \mathcal K(\mathbb{O}_s\otimes\mathbb{O}_s)\cong\mathfrak{e}_{8,8}\cong  \mathcal K(\mathbb{O}\otimes\mathbb{O})\,;\quad
  \mathcal K(\mathbb{O}\otimes\mathbb{O}_s)\cong\mathfrak{e}_{8,-24}.
     $$ 
\end{corollary}   

\begin{proof}
By \cref{pr_idsE6}, the inner ideals of $\mathfrak{e}_{6,-26}$ have dimensions $8$ and $16$, 
the inner ideals of $\mathfrak{e}_{6,-14}$ have dimensions $1$ and $8$ and 
the inner ideals of $\mathfrak{e}_{6,2}$ have dimensions $1$, $2$, $3$ and $8$.
The inner ideals of $\mathfrak{e}_{6,6}$ have dimensions  $1$, $2$, $3$, $4$, $5$, $8$ and $16$ by \cref{eq_idsE6}.

Take $\mathcal A=\mathbb{O}\otimes\mathbb{C}_s$, with skew-dimension $8$. By \cref{pr_Jeroen2}, the algebra 
$\mathcal K(\mathcal A)$ has an inner ideal of dimension $16$, so that its signature is either $6$ or $-26$. But it is not $6$ since the Witt index of $q_{\mathcal A}$ is $0$, so \cref{pr_Jeroen1} says that $\mathcal K(\mathcal A)$ has not non-zero inner ideals contained properly in that one of dimension $8$ (the split algebra has, on the contrary).

If $\mathcal A=\mathbb{O}_s\otimes\mathbb{C}_s$, the Witt index of $q_{\mathcal A}$ is $4$, so $\mathcal K(\mathcal A)$ has inner ideals of dimensions $1$, $2$, $3$, $4$ contained in $\mathcal K_2$ by \cref{pr_Jeroen1}. Only the split real form satisfies that condition.

If $\mathcal A=\mathbb{O}\otimes\mathbb{C}$, the Witt index of $q_{\mathcal A}$ is $1$ and $\mathcal K(\mathcal A)$ cannot have inner ideals contained properly in $\mathcal K_2$ (of dimension 8) up to those of dimension 1. This forces the signature $-14$.

Now, the Witt index of the Albert form of $\mathbb O_s\otimes\mathbb{C}$ is $3$, and so its Kantor construction gives a real form of signature 2, by similar arguments. This finishes the $E_6$-case. \smallskip

For $E_7$, we only have  to note that      
  the Albert form has  $I_{5,5}$ as a related matrix (in a suitable basis) for $\mathcal A=\mathbb{O}_s\otimes\mathbb{H}_s$, 
$I_{3,7}$ for $\mathcal A=\mathbb{O}_s\otimes\mathbb{H}$, 
$I_{7,3}$ for $\mathcal A=\mathbb{O}\otimes\mathbb{H}$, and
$I_{9,1}$ for $\mathcal A=\mathbb{O}\otimes\mathbb{H}_s$, with Witt indices $5$, $3$, $3$ and $1$ respectively.
\cref{pr_Jeroen1} gives immediately the signature of the Lie algebras constructed from these structurable algebras, once we recall that 
$\mathfrak{e}_{7,-25}$ has only inner ideals of dimensions $1$, $10$ and $27$,  
$\mathfrak{e}_{7,5}$ has   inner ideals of dimensions $1$, $2$, $3$ and $10$, while the split real form has an inner ideal of dimension $5$.   Finally, the $E_8$-case is immediate once one notes that the Witt indices of $C_1\otimes C_2$, with $C_1,C_2\in\{\mathbb{O},\mathbb{O}_s\}$, are $7$ if $C_1=C_2$ and $3$ otherwise.
\end{proof}

\begin{remark}
Having these realizations of the real forms of $E_6$ helps to explain some of their fine gradings, described in \cite{Draper26} and \cite{Draper14}. For instance, an immediate consequence is that all the non-compact real forms of $\mathfrak{e}_{6}$, namely,
$\mathfrak{e}_{6, 6}$, $\mathfrak{e}_{6, 2}$, $\mathfrak{e}_{6, -14}$ and $\mathfrak{e}_{6, -26}$, possess a  $\ZZ\times \ZZ_2^4$-grading which does not admit proper refinements. Similarly, all the non-compact real algebras of type $E_7$ have a  fine $\ZZ\times \ZZ_2^5$-grading and all the non-compact real algebras of type $E_8$ have a  fine $\ZZ\times \ZZ_2^6$-grading. This is clear from two facts:   the complexifications of these gradings are fine \cite{librogradings}, and the $\ZZ$-grading given by the Kantor construction is compatible with the gradings on the composition algebras, where the real composition algebras of dimension $2^k$ are always $\ZZ_2^k$-graded. Note that, in spite of the active work   trying to find all the fine gradings on the simple Lie algebras, see for instance the monograph \cite{librogradings}, there is still quite some work to do in the exceptional real Lie algebras.
\end{remark}

Many of the inner ideals of the real exceptional Lie algebras have appeared as particular cases of Propositions~\ref{pr_Jeroen1} and \ref{pr_Jeroen2}, as we have seen through the above proof. In fact, most of them appear in this way, according to the classification obtained in Section~\ref{section_Explicit classification}. We would like to find also descriptions of the remaining inner ideals which are not in terms of roots, but in terms of a model or explicit construction of the algebra, similarly to the descriptions above. This follows being our philosophy through the next subsections.

\subsection{Inner ideals coming from Jordan algebras}\label{sec_Jor}

If $\mathcal J$ is a Jordan algebra, then $\mathcal{K}_{1}=\mathcal J$ is an abelian inner ideal of $\mathcal{K}(\mathcal J)= \mathcal{K}_{-1}\oplus\mathcal{K}_{0}\oplus\mathcal{K}_{1} $, since it is the corner of a $\ZZ$-grading (the involution is the identity and $\mathcal S=0$). Moreover, the inner ideals of the Jordan algebra provide inner ideals of the related Lie algebra. Recall that a subspace $B\le \mathcal J$ is said an \emph{inner ideal} if $U_B(\mathcal J)\subseteq B$, for the quadratic operator $U_x=U_{x,x}=2L_x^2-L_{x^2}$.

\begin{lemma}\label{le_innerJordan}
The   only inner ideals of  $ \mathcal K( \mathcal J)$ contained in the component $\mathcal K_1$ are the sets
  $\{(b,0):b\in B\} $ for some   inner ideal  $B$   of $\mathcal J$.
\end{lemma}

\begin{proof}
This is evident:  $I=\{(b,0):b\in B\} $ is an inner ideal of  $ \mathcal K( \mathcal J)$ if and only if   $[I,[I,\mathcal K_{-1}]]\subseteq I$ holds, taking   the grading into consideration. But 
$$
{[}(a,0),[(b,0),(c,0)\tilde{\ } ]]=(-V_{b,c}(a),0)=(-U_{b,a}(c),0)
$$
belongs to $I$ if and only if  $U_{B,B}(\mathcal  J)\subseteq B$.
\end{proof}

Take $\mathcal J=H_3(C)$ the Jordan algebra of Hermitian matrices of size 3 with coefficients in the complex Cayley algebra $C$ with the symmetrized product $a\circ b=\frac{ab+ba}2$, the so called 
 \emph{Albert algebra}. The Lie algebra obtained when applying the TKK-construction is the complex Lie algebra of type $E_7$ \cite[Proposition~16]{libroJacobson2}. Thus,   an inner ideal of dimension 27 
 of   the complex Lie algebra $\mathfrak{e}_7$ is given by the 1-part of the 3-grading, which can be identified with this Albert algebra.
 
  If we consider any real form of the Albert algebra, its TKK-construction will give a real form of the complex algebra $\mathfrak{e}_7$, also with an inner ideal of dimension $27$. In particular, the TKK-construction applied to  $H_3(\OO_s)$ and $H_3(\OO)$  has to give the real forms of signature either $7$ or $-25$  according to \cref{se_tipoe7}, because these real forms are the only ones with inner ideals of dimension 27. In order to distinguish the obtained signature, we could check that the construction used by Jacobson in \cite[Table~(143), row b)]{libroJacobson2} coincides with ours, what would imply $\mathcal{K}(H_3(\OO_s))\cong\mathfrak{e}_{7,7}$ and $\mathcal{K}(H_3(\OO))\cong\mathfrak{e}_{7,-25}$. But note that we can alternatively use our knowledge on the remaining inner ideals. 
Take an idempotent $e\in\OO_s$ and, as in \cite[Main Theorem, iv)]{innerJordan}, consider 
 $$
 B=\left\{\begin{pmatrix}\lambda&a&\nu e\\
 \bar a&0&0\\\nu\bar e&0&0\end{pmatrix}:\lambda,\nu\in\RR,a\in e\OO_s\right\}.
 $$
 Then $B$ is an inner ideal of dimension $6$ of the Jordan algebra $H_3(\OO_s)$ which provides a $6$-dimensional inner ideal of $\mathcal{K} (H_3(\OO_s))$.
 Consequently $\mathcal{K} (H_3(\OO_s))$ is the split real form $\mathfrak{e}_{7,7}$. This 6-dimensional inner ideal is formed uniquely by extremal elements (clear from \cite{Draper2012}, since an element is extremal if it is after complexification). But $\mathcal{K} (H_3(\OO))$ cannot be the split real form, because \cite[Theorem~8]{innerJordan} says that the inner ideals of $H_3(\OO)$ formed by extremal elements are necessarily one-dimensional. By \cref{le_innerJordan}, $\mathcal{K} (H_3(\OO))$ does not have a six-dimensional inner ideal. Hence, its signature is $-25$.\smallskip

 Now we deal with the inner ideals not appeared so far, which    all satisfy the following property: all their non-zero elements are extremal.
  
 
\subsection{Point line spaces (and some more linear models)}\label{sec_ps}

In order to describe the remaining inner ideals, we will use some ad-hoc constructions of the exceptional (split) Lie algebras based on linear and multilinear algebra, following the philosophy of \cite{modelosF4}.

\begin{definition}
A non-zero subspace $P$ of a Lie algebra $\g$  is said to be a \emph{point space} if $[P,P]=0$ and every  non-zero element in $P$ is an extremal element. (Recall Definition~\ref{def_extremal element}.)
\end{definition}

Point spaces were introduced in \cite{Benkart2009}, in order to complete the missing cases in \cite{inner76}. The following  is hence well-known.

\begin{lemma}
If $P$ is a point space of $\g$, then $P$ is an abelian inner ideal and every subspace of $P$ so is.
\end{lemma}

\begin{proof} Let us take $0\ne x,y\in P$ and $z\in\g$ and let us check that $[x,[y,z]]\in P$. As $[x,y]=0$, the Jacobi identity implies that $[x,[y,z]]=[y,[x,z]]$ so that 
$$
 [x,[y,z]]=\frac12\big([x,[y,z]]+[y,[x,z]]\big)=\frac12\big([x+y,[x+y,z]]-[x,[x,z]]-[y,[y,z]]\big),  
$$
which belongs to $ \langle x,y\rangle\subseteq P$ because $x$, $y$ and $x+y$ are extremal elements (except for  $x+y=0$,  in which case the result is   clear too).
Thus  $[P,[P,\g]]\subseteq P$. 

Moreover, it is evident that every non-zero  subspace of $P$ is also a point space, and hence an abelian inner ideal. 
\end{proof}

Now we will describe a point space of dimension 8 of $\mathfrak{e}_{8,8}$, a point space of dimension 7 of $\mathfrak{e}_{7,7}$ and a point space of dimension 5 of $\mathfrak{e}_{6,6}$. The fact that they are point spaces appears explicitly in  \cite{Draper2012}. These point spaces and their subspaces will exhaust all the abelian inner ideals of the split exceptional Lie algebras not appeared in Subsections~\ref{sec_tensorpro} and \ref{sec_Jor}. 


\subsubsection{A linear model of $\mathfrak{e}_{8,8}$}  

Take $U$ a complex or real vector space of dimension $8$ and consider the graded vector space
$$
\g=\g_{-3}\oplus\g_{-2}\oplus\g_{-1}\oplus\g_{0}\oplus\g_{1}\oplus\g_{2}\oplus\g_{3}
$$
given by  
\begin{equation}\label{eq_linearmodele8}
\g_{3}=U;\ 
\g_{2}=\Lambda^6U;\ 
\g_{1}=\Lambda^3U;\ 
\g_{0}=\mathfrak{gl}(U);\ 
\g_{-1}=\Lambda^5U;\ 
\g_{-2}=\Lambda^2U;\ 
\g_{-3}=\Lambda^7U.
\end{equation}
Now $\g$ can be endowed with a Lie algebra structure of type $E_8$ (over $\CC$ or $\RR$, respectively). Indeed, think first of the complex 
Lie algebra $  L$ of type $E_8$, $H$ a Cartan subalgebra, $  L=H\oplus(\oplus_{\alpha\in \Phi}  L_\alpha)$ the root decomposition relative to $H$, and $\{\alpha_1,\dots,\alpha_8\}$ a set of simple roots (with labelling as in Section~\ref{se_losdiagramasconcretos}). There is a $\ZZ$-grading on $  L$ with homogeneous components given by:
$$\begin{array}{l}
  L_n=\oplus\{  L_\alpha:\alpha=\sum_{i=1}^8p_i\alpha_i,\  p_2=n\},\vspace{2pt}\\
  L_0=H\oplus \big(\oplus\{  L_\alpha:\alpha=\sum_{i=1}^8p_i\alpha_i,\   p_2=0\} \big).
\end{array}
$$
A glimpse of the roots gives us that this grading has 7 pieces and the related dimensions are $8$, $28$, $56$, $64$, $56$, $28$ and $8$ respectively (necessarily $  L_n$ and $  L_{-n}$ are dual one of each other through the Killing form $\kappa$). This fits with the fact that $  L_0$ is a reductive subalgebra which is the sum of a  one-dimensional center and a semisimple subalgebra whose Dynkin diagram is obtained when removing the $2^{\mathrm{nd}}$-node of the Dynkin diagram of $E_8$, that is, a simple Lie algebra of type $A_7$. 
Thus $  L_0$  is a Lie algebra isomorphic to $\mathfrak{gl}(8,\CC)$ and to $\mathfrak{gl}(U)=\g_0$ when our $8$-dimensional vector space is considered over the complex numbers. Also, according to \cite[Chapter~8]{libroKac}, $   L_{n}$ is an irreducible $  L_{0}$-module.
By dimension count, for each $n$ there is an obvious identification among $  L_{n}$ and either $\g_n$ or $\g_{-n}$   compatible with the actions of $  L_{0}$ and $\g_0$ respectively. 
Changing $U$ by its dual if necessary, we can assume that  the $L_0$-module $L_1$  is isomorphic to the  $\mathfrak{gl}(U)$-module $\Lambda^3U=\g_1$.
The fact that $\dim_\CC\hom(  L_{n}\otimes  L_{m}, L_{n+m})=1$ (\cite[Chapter~8]{libroKac} too) gives our choices for the remaining exterior powers of $U$. 
We mean, the $L_0$-modules $L_n$  are now isomorphic to the  $\mathfrak{gl}(U)$-modules $\g_n=\Lambda^{[3n]_8}U$ where $[n]_8$ denotes here $n$ modulo $8$.
Now, the isomorphism of vector spaces between $  L$ and $\g$ induces a structure of (complex) Lie algebra on $\g$. It is clear that the bracket $[\g_{n},\g_{m}]\subseteq \g_{n+m}$ is determined up to scalar by the only $\mathfrak{gl}(U)$-invariant map $\Lambda^{[3n]_8}U\times\Lambda^{[3m]_8}U\to \Lambda^{[3(n+m)]_8}U$ if $n+m\ne0$. And, if $n+m=0$, determined up to scalar by the only $\mathfrak{gl}(U)$-invariant map $\Lambda^{[3n]_8}U\times\Lambda^{[3n]_8}U^*\to \mathfrak{gl}(U)$.
The concrete scalars can be determined (once we have fixed such invariant maps) by imposing the Jacobi identity to $(\g,[\ ,\ ])$ instead of passing through the explicit isomorphism, as in \cite{modelosF4}.  
In any case, it is possible to have more information on the pieces of this grading without computing the concrete scalars. 

\begin{lemma}
$\g_{3}=U$ is a point space of $\g$.
\end{lemma}

As we know from  \cite{Draper2012} that any inner ideal of dimension 8 should be a point space, and $\g_{3}$ is an inner ideal because it is the corner of a $\mathbb{Z}$-grading, this lemma would not require a proof. But we think that a direct argument provides deeper insight.

\begin{proof}
If we think of $\g_{-3}$ as $U^*$, take $0\ne u\in U $ and $\alpha\in U^*$ and let us check that $u$ is an extremal element: $[u,[u,\g]]\subseteq \CC u$. Taking the grading into consideration, it is enough to prove that $[u,[u,\g_{-3}]]\subseteq \CC u$, since $[u,[u,\g_n]]\subseteq \g_{n+6}$ vanishes when $n\ne -3$. Up to scalar, $[u,\alpha]$ corresponds to the map in $\mathfrak{gl}(U)$ given by $v\mapsto \alpha(v)u$ for any $v\in U$. As the action of $\mathfrak{gl}(U)$ on $\g_3=U$ is the natural action, also up to scalar $[u,[u,\alpha]]$  should coincide with $-\alpha(u)u\in\CC u$.
\end{proof}

If we now  consider $U$ as a real vector space, the  same construction \eqref{eq_linearmodele8} gives a real Lie algebra of type $E_8$ which possesses  an inner ideal of dimension $8$ (and also inner ideals of each dimension less than 8), so that the description of inner ideals in \cref{sub_e8} implies that the constructed real form $\g$ is necessarily of split type, $\mathfrak{e}_{8,8}$. 

 
\subsubsection{A linear model of $\mathfrak{e}_{7,7}$}

Similarly to the previous case, the $\ZZ$-grading  on the complex Lie algebra $L$ of type $E_7$ (and also, on the real Lie algebra  $\mathfrak{e}_{7,7}$) produced by choosing the second node is 
$$\begin{array}{l}
   L_n=\oplus\{ L_\alpha:\alpha=\sum_{i=1}^7p_i\alpha_i, \  p_2=n\},\vspace{2pt}\\
 L_0=H\oplus\big(\oplus\{ L_\alpha:\alpha=\sum_{i=1}^7p_i\alpha_i, \  p_2=0\} \big).
\end{array}
$$
This is a $5$-grading, since the maximal root   $2\alpha_1+2\alpha_2+3\alpha_3+4\alpha_4+3\alpha_5+2\alpha_6+\alpha_7$ has 2 as the coefficient of $\alpha_2$. Looking at how many roots have coefficient of $\alpha_2$ equal to $\pm2,\pm1,0$, we get dimensions of the homogeneous components equal to $7$, $35 $ and $49$, respectively. Also, looking at the Dynkin diagram obtained when removing the second node, we know $ L_0\cong\mathfrak{gl}(7,\CC)$. So there is a Lie algebra structure on the following  $\ZZ$-graded vector space.
Take $U$ a vector space of dimension 7, over the complex numbers or over the real ones, and consider
$$
\g= \g_{-2}\oplus\g_{-1}\oplus\g_{0}\oplus\g_{1}\oplus\g_{2} 
$$
given by  
\begin{equation*}\label{eq_linearmodele7}
\g_{2}=U,\quad 
\g_{1}=\Lambda^4U,\quad 
\g_{0}=\mathfrak{gl}(U),\quad
\g_{-1}=\Lambda^3U,\quad 
\g_{-2}=\Lambda^6U,
\end{equation*}
with the above explained Lie algebra structure induced  by the isomorphism with $ L$.
The same arguments as above yield 

\begin{lemma}
The vector space of dimension $7$ given by   
$\g_{2}=U$ is a point space of $\g$.
\end{lemma}

This corresponds with the inner ideal named $\mathcal{B}_{\{2\}}$ of $\mathfrak{e}_{7,7}$, and their proper  subspaces of dimension at least $4$  correspond to the inner ideals $\mathcal{B}_{\{2,7\}}\supseteq\mathcal{B}_{\{2,6,7\}}\supseteq\mathcal{B}_{\{2,5,6,7\}}$, so that this completes, joint with 
the inner ideals   appeared above, the list of inner ideals of $\mathfrak{e}_{7,7}$.

\begin{remark}
In this case we can give an alternative description of an inner ideal of $\mathfrak{e}_{7,7}$ which is a point space of dimension 7.
Recall that, as the $\ZZ$-grading has 5 pieces, it could be the TKK-construction of some structurable algebra. 
This is indeed the case,   the structurable algebra of dimension $35$ is the Smirnov algebra \cite{Smirnov35}. Thus, the point space of dimension 7 can be understood as the set of the skew-symmetric elements of the Smirnov algebra.
\end{remark}

\subsubsection{A linear model of $\mathfrak{e}_{6,6}$}

In this case, the considered $\ZZ$-grading  on the complex Lie algebra of type $E_6$ (and also, on the real Lie algebra  $\mathfrak{e}_{6,6}$) is produced by choosing the third labelled node. So, after choosing a Cartan subalgebra and a set of simple roots $\{\alpha_i\}_{i=1}^6$ of the related root system, the homogeneous components are 
$$\begin{array}{l}
 L_n=\oplus\{ L_\alpha:\alpha=\sum_{i=1}^6p_i\alpha_i, \  p_3=n\},\vspace{2pt}\\
 L_0=H\oplus\big(\oplus\{ L_\alpha:\alpha=\sum_{i=1}^6p_i\alpha_i, \  p_3=0\} \big).
\end{array}
$$
This decomposition $L=\oplus_{n\in\ZZ} L_n$ is again a $5$-grading, since the coefficient of $\alpha_3$ of the maximal root $\alpha_1+2\alpha_2+2\alpha_3+3\alpha_4+2\alpha_5+\alpha_6$ is $2$. The $0$-part is the sum of a one-dimensional center and a semisimple algebra, which is the sum of two simple ideals of types $A_4$ and $A_1$. The dimensions of the non-neutral  homogeneous components, by counting roots, are $5$ and $20$.
The considerations about their irreducibility   as $L_0$-modules   work as in the previous cases, although in this case an irreducible module for a semisimple algebra (the derived algebra $[L_0,L_0]$) has to be a tensor product of irreducible modules for each simple component. A model adapted to this situation is the following.

Take $U$ a real vector space of dimension $5$ (not 6) and $V$ a real vector space of dimension $2$. Take
$$
\g= \g_{-2}\oplus\g_{-1}\oplus\g_{0}\oplus\g_{1}\oplus\g_{2} 
$$
given by  
\begin{equation*}\label{eq_linearmodele6}
\g_{2}=U\otimes \RR,\quad 
\g_{1}=\Lambda^3U\otimes V,\quad 
\g_{0}=\mathfrak{gl}(U)\oplus \mathfrak{sl}(V),\quad
\g_{-1}=\Lambda^2U\otimes V,\quad 
\g_{-2}=\Lambda^4U\otimes \RR,
\end{equation*}
with the Lie algebra structure such that $\g$ is a simple real Lie algebra (the uniqueness of the bracket between components is again obtained by passing to the complexification and using then the results in  \cite[Chapter~8]{libroKac}). 
We prefer to write $\g_{2}=U\otimes \RR$ instead the nicer but equivalent expression $\g_{2}=U$ for emphasizing that the action of $\mathfrak{sl}(V)$ on   $\g_2$ is trivial. 
Once more we have

\begin{lemma}
The vector space of dimension $5$ given by 
$\g_{2}=U\otimes \RR$ is a point space of $\g$.
\end{lemma}

The argument follows being that the only invariant map $U\otimes U^*\to \mathfrak{gl}(U)$ is given by $u\otimes\alpha\mapsto(v\mapsto \alpha(v)u)$, up to scalar multiple.  Alternatively, $\g_{2}$ is an inner ideal since it is a corner of a $\ZZ$-grading, which should be  a point space because every inner ideal of dimension $5$ of $E_6$ is a point space.

This point space cannot be described as the set of skew-symmetric elements of a structurable algebra, since there is no  central  simple structurable algebra with skew-dimension $5$ (a simple case-by-case computation from \cite[Theorem~11]{Allison1979}). 

One of the applications of this viewpoint is to provide models to describe the gradings of the exceptional Lie algebras, so, in some sense, a method to look for symmetries. Consult \cite{typeE} in this line, where Kantor's construction is applied to a structurable algebra of dimension 20 to get a $\ZZ$-grading on $\mathfrak{e}_6$ and on some of its real forms, but such algebra structurable algebra has skew-dimension $1$ (not $5$).

\subsection{Some comments on Jordan pairs}\label{se_JP}

For completeness, we relate our results with  Jordan pairs. 
Real simple Jordan pairs are classified in \cite[\S11.4]{libroLoos}, where the study of real bounded symmetric domains is reduced to that one of real   Jordan pairs.
As is explained in \cite[\S12.5]{libroAntonio}, abelian  inner ideals of Lie algebras and inner ideals of Jordan pairs are essentially the same mathematical object. More precisely, 
 if $B$ is an abelian inner ideal of a Lie algebra $L$, then $\mathrm{Sub}_LB:=V=(V^+,V^-)$, for $V^+=B$ and 
$V^-=L/\mathrm{Ker}_LB$, is a Jordan pair called the \emph{subquotient} of $B$. Recall that $\mathrm{Ker}_LB:=\{x\in L:[B,[B,x]]=0\}$. Now the inner ideals of $L$ contained in $B$ are the inner ideals of $V$ contained in $V^+$ \cite[Proposition~11.47i)]{libroAntonio}. We look at the exceptional cases in the light of Jordan pairs. Recall that any Jordan pair has attached a Lie algebra by a construction also called   TKK-construction.

In the complex case there is (up to conjugation) only one maximal inner ideal  of the complex Lie algebra of type $E_6$, with dimension $16$. It has as subquotient  
the \emph{Bi-Cayley pair} $V=(\mathrm{Mat}_{1\times2}(C),\mathrm{Mat}_{2\times1}(C))$ for $C$ the complex Cayley algebra, where the $Q$-operator is $Q_ab=(ab)a$. 
Moreover, the TKK-construction  applied to the Bi-Cayley pair is again the Lie algebra of type $E_6$. 
In the real case there are two real forms of the Bi-Cayley pair. The  TKK-construction of the Jordan pair 
$V=(\mathrm{Mat}_{1\times2}(\OO),\mathrm{Mat}_{2\times1}(\OO))$ gives $\mathfrak{e}_{6,-26}$, since $\mathfrak{e}_{6,2}$ and $\mathfrak{e}_{6,-14}$ have no inner ideals of dimension $16$, while $\mathfrak{e}_{6,6}$ has inner ideals of dimensions from $1$ to $5$, in particular it has point spaces of dimension $1$, which is not the case for $V$. On the other hand, it is clear that   the TKK-construction of the Jordan pair $(\mathrm{Mat}_{1\times2}(\OO_s),\mathrm{Mat}_{2\times1}(\OO_s))$ gives $\mathfrak{e}_{6,6}$, again taking into account the lattice of inner ideals. 
For the other real forms, recall that if $q\colon X\to\FF$ is a non-degenerate quadratic form on a vector space $X$ over a field $\FF$, then $(X,X)$ is Jordan pair called  \emph{Clifford pair}, where the $Q$-operator is given by $Q_xy=q(x,y)x-q(x)y$. We refer to this Jordan pair as $\mathcal Q_n$ ($n=\dim X$)  when there is no ambiguity in the signature of $q$. 
The lattice of the proper inner ideals of $\mathfrak{e}_{6,2}$ (resp. $\mathfrak{e}_{6,-14}$) coincides with the lattice of inner ideals of the Clifford Jordan pair defined by a quadratic form of Witt index $3$ (resp. $1$) on a real vector space of dimension $8$.

The complex Lie algebra of type $E_7$ can be seen as the TKK-construction of the \emph{Albert pair} $V=(\mathcal J,\mathcal J)$ for $\mathcal J=H_3(C)$ and $C$ the complex Cayley algebra. There are two  maximal inner ideals, one with dimension  $27$ and the other one a point space of dimension $7$. The subquotient of  the first one is the Albert pair. Similar arguments to \cref{sec_Jor} allow us to conclude that $\mathfrak{e}_{7,-25}$ (resp. $\mathfrak{e}_{7,7}$) can be obtained as the TKK-construction of the Albert pair given by $H_3(\OO)$ (resp. $H_3(\OO_s)$). Lastly, the lattice of the proper inner ideals of $\mathfrak{e}_{7,5}$ coincides with that one of   the Clifford Jordan pair defined by a quadratic form of Witt index $3$   on a real vector space of dimension $10$. 

The complex Lie algebra of type $E_8$ has   two maximal inner ideals up to conjugation,  one whose subquotient is the Clifford pair $\mathcal Q_{14}$ and the other one a point space of dimension $8$. This time the
 TKK-construction applied to this Clifford pair $\mathcal Q_{14}$ does not give the full Lie algebra of type $E_8$ but only a subalgebra of type $D_8$. The lattice of $\mathfrak{e}_{8,-24}$ coincides with the lattice of inner ideals of the Clifford Jordan pairs defined by a quadratic form of Witt index $3$
 on a real vector space of dimension $14$.

 
\subsection{Conclusions}

In the above subsections, we have found explicit descriptions of the inner ideals of the exceptional real Lie algebras, with the exception of those of type $G_2$, which are better understood.   These inner ideals, closely related to the different constructions of the exceptional Lie algebras, permit us to have a taste of the general behavior of the inner ideals. In a next paper, we will deal with some incidence geometries related to the inner ideals of the simple real algebras. 
We will take advantage of the very precise information obtained here.
There will appear  polar spaces, root shadows and some kinds of partial linear spaces. In some cases, the minimal inner ideals will constitute a Moufang set. 
These nice geometries will explain, in some way, the more involved structure of the inner ideals or of the extremal elements instead of considering them \emph{up to automorphism}, which often provide an insufficient answer. We mean, sometimes the \emph{vertical} distribution of the inner ideals, once we have chosen them compatible with a determined Cartan subalgebra, hides in some way a more intricate structure, which is far from being completely described only by saying that they are conjugated by an automorphism.

\bibliographystyle{alpha}
\bibliography{InnerIdealsRealAlgebras} 

\begin{thebibliography}{DFLGGL12}

\bibitem[All78]{Allison78}
B.~N. Allison.
\newblock A class of nonassociative algebras with involution containing the
  class of {J}ordan algebras.
\newblock {\em Math. Ann.}, 237(2):133--156, 1978.

\bibitem[All79]{Allison1979}
B.~N. Allison.
\newblock Models of isotropic simple {L}ie algebras.
\newblock {\em Communications in Algebra}, 7(17):1835--1875, 1979.

\bibitem[All88]{Allison1988}
B.~N. Allison.
\newblock Tensor products of composition algebras, {A}lbert forms and some
  exceptional simple {L}ie algebras.
\newblock {\em Trans. Amer. Math. Soc.}, 2(306):667--695, 1988.

\bibitem[Ben74]{tesisgeorgia}
G.~Benkart.
\newblock {\em Inner Ideals and The structure of {L}ie algebras}.
\newblock PhD thesis, Yale University, 1974.

\bibitem[Ben76]{inner76}
G.~Benkart.
\newblock The {L}ie inner ideal structure of associative rings.
\newblock {\em J. Algebra}, 43(2):561--584, 1976.

\bibitem[Ben77]{transactions77}
G.~Benkart.
\newblock On inner ideals and ad-nilpotent elements of {L}ie algebras.
\newblock {\em Transactions of the American Mathematical Society}, 232:61--81,
  1977.

\bibitem[BFL09]{Benkart2009}
G.~Benkart and A.~Fern\'{a}ndez~L\'{o}pez.
\newblock The {L}ie inner ideal structure of associative rings revisited.
\newblock {\em Communications in Algebra}, 37(11):3833--3850, 2009.

\bibitem[BR13]{Baranov2013}
A.~A. Baranov and J.~Rowley.
\newblock Inner ideals of simple locally finite {L}ie algebras.
\newblock {\em J. Algebra}, 379:11--30, 2013.

\bibitem[CF18]{Cuypers2018}
H.~Cuypers and Y.~Fleischmann.
\newblock A geometric characterization of the classical {L}ie algebras.
\newblock {\em Journal of Algebra}, 502:1--23, 2018.

\bibitem[CI06]{Cohen2006}
A.~M. Cohen and G.~Ivanyos.
\newblock Root filtration spaces from {L}ie algebras and abstract root groups.
\newblock {\em J. Algebra}, 300(2):433--454, 2006.

\bibitem[CI07]{Cohen2007}
A.~M. Cohen and G.~Ivanyos.
\newblock Root shadow spaces.
\newblock {\em European Journal of Combinatorics}, 28(5):1419--1441, 2007.

\bibitem[CIR08]{Cohen2008}
A.~M. Cohen, G.~Ivanyos, and D.~Roozemond.
\newblock Simple {L}ie algebras having extremal elements.
\newblock {\em Indag. Math. (N.S.)}, 19(2):177--188, 2008.

\bibitem[CM21]{Cuypers2020}
H.~Cuypers and J.~Meulewaeter.
\newblock Extremal elements in {L}ie algebras, buildings and structurable
  algebras.
\newblock {\em J. Algebra}, 580:1--42, 2021.

\bibitem[Coh21]{Cohen21}
A.~M. Cohen.
\newblock Inner ideals in {L}ie algebras and spherical buildings.
\newblock {\em Indag. Math. (N.S.)}, 32(5):1115--1138, 2021.

\bibitem[DE14]{typeE}
C.~Draper and A.~Elduque.
\newblock Fine gradings on the simple {L}ie algebras of type {$E$}.
\newblock {\em Note Mat.}, 34(1):53--88, 2014.

\bibitem[DFG16]{DraperFontanals2016}
C.~Draper~Fontanals and V.~Guido.
\newblock On the real forms of the exceptional {L}ie algebra {$\mathfrak{e}_6$}
  and their {S}atake diagrams.
\newblock In {\em Non-associative and non-commutative algebra and operator
  theory}, volume 160 of {\em Springer Proc. Math. Stat.}, pages 211--226.
  Springer, Cham, 2016.

\bibitem[DFLGGL12]{Draper2012}
C.~Draper, A.~Fern\'{a}ndez~L\'{o}pez, E.~Garc\'{\i}a, and M.~G\'{o}mez~Lozano.
\newblock The inner ideals of the simple finite dimensional {L}ie algebras.
\newblock {\em Journal of Lie Theory}, 22(4):907--929, 2012.

\bibitem[DG16]{Draper26}
C.~Draper and V.~Guido.
\newblock Gradings on the real form {$\mathfrak{e}_{6,-26}$}.
\newblock {\em J. Math. Phys.}, 57(10):18pp, 2016.

\bibitem[DG18]{Draper14}
C.~Draper and V.~Guido.
\newblock Gradings on the real form {$\mathfrak{e}_{6,-14}$}.
\newblock {\em J. Math. Phys.}, 59(10):20 pp, 2018.

\bibitem[Djo82]{RealesZgrads}
D.~{\v{Z}}. Djokovi{\v{c}}.
\newblock Classification of {$\mathbb Z$}-graded real semisimple {L}ie
  algebras.
\newblock {\em J. Algebra}, 76(2):367--382, 1982.

\bibitem[DMM]{Medts2020}
T.~De~Medts and J.~Meulewaeter.
\newblock Inner ideals and structurable algebras: {M}oufang sets, triangles and
  hexagons.
\newblock {\em Israel Journal of Mathematics {(to appear)}}.

\bibitem[Dra08]{modelosF4}
C.~Draper.
\newblock Models of the {L}ie algebra {$F4$}.
\newblock {\em Linear Algebra Appl.}, 428(11-12):2813--2839, 2008.

\bibitem[EK13]{librogradings}
A.~Elduque and M.~Kochetov.
\newblock {\em Gradings on simple Lie algebras}.
\newblock Mathematical Surveys and Monographs, 189. American Mathematical
  Society, Providence, RI, 2013.

\bibitem[Fau73]{Faulkner1973}
J.~R. Faulkner.
\newblock On the geometry of inner ideals.
\newblock {\em Journal of Algebra}, 26:1--9, 1973.

\bibitem[FL19]{libroAntonio}
A.~Fern\'andez~L\'opez.
\newblock {\em Jordan structures in {L}ie algebras}, volume Mathematical
  Surveys and Monographs, 240 of {\em Graduate Studies in Mathematics}.
\newblock American Mathematical Society, Providence, RI, 2019.

\bibitem[FLGGLN07]{Lopez2007}
A.~Fern\'{a}ndez~L\'{o}pez, E.~Garc\'{\i}a, M.~G\'{o}mez~Lozano, and E.~Neher.
\newblock A construction of gradings of {L}ie algebras.
\newblock {\em International Mathematics Research Notices. IMRN}, (16):Art. ID
  rnm051, 34, 2007.

\bibitem[GGL09]{Miguel2009}
E.~Garc\'{\i}a and M.~G\'omez~Lozano.
\newblock A note on a result of {K}ostrikin.
\newblock {\em Comm. Algebra}, 37(7):2405--2409, 2009.

\bibitem[Hel01]{Helgason2001}
S.~Helgason.
\newblock {\em Differential geometry, {L}ie groups, and symmetric spaces},
  volume~34 of {\em Graduate Studies in Mathematics}.
\newblock American Mathematical Society, Providence, RI, 2001.
\newblock Corrected reprint of the 1978 original.

\bibitem[Jac58]{Jacobson1958}
N.~Jacobson.
\newblock A note on three dimensional simple {L}ie algebras.
\newblock {\em J. Math. Mech.}, 7:823--831, 1958.

\bibitem[Jac71]{libroJacobson2}
N.~Jacobson.
\newblock {\em Exceptional Lie algebras}.
\newblock Lecture Notes in Pure and Applied Mathematics. Marcel Dekker, Inc.,
  1971.

\bibitem[Jac79]{libroJaconsonLie}
N.~Jacobson.
\newblock {\em Lie algebras}.
\newblock Dover Publications, Inc., 1979.
\newblock Republication of the 1962 original.

\bibitem[Kac90]{libroKac}
V.~G. Kac.
\newblock {\em Infinite dimensional Lie algebras}.
\newblock Cambridge University Press, 1990.

\bibitem[Loo77]{libroLoos}
O.~Loos.
\newblock {\em Bounded Symmetric Domains and Jordan Pairs}.
\newblock University of California at Irvine, Department of Mathematics:
  Mathematical Lectures, 1977.

\bibitem[McC71]{innerJordan}
K.~McCrimmon.
\newblock Inner ideals in quadratic {J}ordan algebras.
\newblock {\em Trans. Amer. Math. Soc.}, 159:445--468, 1971.

\bibitem[Smi90]{Smirnov35}
O.~N. Smirnov.
\newblock An example of a simple structurable algebra.
\newblock {\em Algebra i Logika}, 29(4):491--499, 1990.

\bibitem[Tit74]{Tits74}
J.~Tits.
\newblock {\em Buildings of spherical type and finite BN-pairs}.
\newblock Lecture Notes in Mathematics, 386. Springer-Verlag, Berlin-New York,
  1974.

\end{thebibliography}
\end{document}